\documentclass[reqno, 12pt]{amsart}

\usepackage[algoruled,lined,noresetcount,norelsize]{algorithm2e}

\usepackage{graphicx}
\usepackage{amsmath,amssymb,amsthm}
\usepackage{mathrsfs}
\usepackage{mathabx}\changenotsign
\usepackage{dsfont}
\usepackage{xcolor}
\usepackage[backref]{hyperref}
\hypersetup{
    colorlinks,
    linkcolor={red!60!black},
    citecolor={green!60!black},
    urlcolor={blue!60!black}
}
\usepackage[T1]{fontenc}
\usepackage{lmodern}
\usepackage[babel]{microtype}
\usepackage[english]{babel}

\usepackage{geometry}
\numberwithin{equation}{section}
\numberwithin{figure}{section}

\usepackage{enumitem}

\theoremstyle{plain}
\newtheorem{thm}{Theorem}[section]

\newtheorem{obs}[thm]{Observation}
\newtheorem{clm}[thm]{Claim}
\newtheorem{cor}[thm]{Corollary}
\newtheorem{lemma}[thm]{Lemma}

\theoremstyle{definition}

\newtheorem{dfn}[thm]{Definition}

\newcommand\eps{\varepsilon}
\newcommand{\Exp}{\mathbb{E}}
\newcommand{\Prob}{\mathbb{P}}
\newcommand{\Ik}{I_k}
\newcommand{\E}{\mathcal{E}}

\newcommand{\done}{{\color{red}\checkmark}}

\def\tref#1{ {\tiny \ref{#1} } }

\def\itm#1{\rm ({#1})}

\def\itmarab#1{\mbox{\itm{{\it #1\,}\arabic{*}\hspace{.05em}}}}

\newcommand{\Tk}{\mathcal{T}_k}
\newcommand{\Tkplus}{\mathcal{T}_{k+1}}

\begin{document}

\title[Connector-Breaker games on random boards]{Connector-Breaker games on random boards}
\author[Dennis Clemens]{Dennis Clemens}
\address{(DC) Technische Universit\"at Hamburg, Institut f\"ur Mathematik, Am Schwarzenberg-Campus 3, 21073 Hamburg, Germany }
\email{dennis.clemens@tuhh.de}
\author[Laurin Kirsch]{Laurin Kirsch}
\address{(LK)}
\email{lkirsch@student.ethz.ch}
\author[Yannick Mogge]{Yannick Mogge}
\address{(YM) Technische Universit\"at Hamburg, Institut f\"ur Mathematik, Am Schwarzenberg-Campus 3, 21073 Hamburg, Germany }
\email{yannick.mogge@tuhh.de}

\maketitle

\begin{abstract}
By now, the Maker-Breaker connectivity game on a complete graph $K_n$ or on a random graph $G\sim G_{n,p}$ is well studied.
Recently, London and Pluh\'ar
suggested a variant in which
Maker always needs to choose her edges in such a way that her graph stays connected.
By their results it follows that for this connected version of the game, the threshold bias  on $K_n$ and the threshold probability on $G\sim G_{n,p}$
for winning the game drastically differ from the corresponding values for the usual Maker-Breaker version, assuming Maker's bias to be 1. However, they observed that the threshold biases of both versions played on $K_n$ are still of the same order if instead Maker is allowed to claim two edges in every round. Naturally, this made London and Pluh\'ar ask whether a similar phenomenon 
can be observed when a $(2:2)$ game is played on $G_{n,p}$. We prove that this is not the case, and determine the threshold probability for winning this game to be of size $n^{-2/3+o(1)}$.
\end{abstract}

%
%

\section{Introduction}

A positional game is a perfect information game played by two players on a {\em board} $X$ equipped with a family of subsets $\mathcal F \subset 2^X$, which represent {\em winning sets}. During each round of such a game both players claim previously unclaimed elements of the board. For instance, in the $(m:b)$~Maker-Breaker variant, Maker and Breaker take turns claiming up to $m$ (as Maker) or up to $b$ (as Breaker) such elements.
Maker wins the game by claiming all elements of a winning set; Breaker wins otherwise. If $m=b=1$, the game is called {\em unbiased}. Otherwise, we call the game {\em biased} with $m$ and $b$ being the respective biases of Maker and Breaker.

Note that Maker-Breaker games are {\em bias monotone} in the sense that claiming more elements of the board never hurts the corresponding player. Given $(X,\mathcal{F})$ and having Maker's bias $m$ fixed, we thus can find an integer $b_0$, called the {\em threshold bias},
such that Breaker wins the $(m,b)$~Maker-Breaker game 
if and only if $b\geq b_0$ holds (except for trivial games, where Maker can win before Breaker's first move).\\

In our paper, we will consider a variant of such Maker-Breaker games played on a graph $G$ sampled according to the {\em Binomial random graph model} $G_{n,p}$ (for short we will write $G\sim G_{n,p}$),
where we fix $n$ vertices and each edge appears with probability $p$ independently of all other choices. It is well known that for monontone increasing graph properties $\mathcal{F}$ this model always comes with a {\em threshold probability} $p^{\ast}$ (see e.g.~\cite{BT1987}) such that
$$
\Prob\left(G\sim G_{n,p}\text{ satisfies }\mathcal{F}\right)
\rightarrow
\begin{cases}
0 & ~ \text{if }p=o(p^{\ast})\\
1 & ~ \text{if }p=\omega(p^{\ast})~ .
\end{cases}
$$
For some properties $\mathcal{F}$ there is even a {\em sharp threshold}
in the sense that
$$
\Prob\left(G\sim G_{n,p}\text{ satisfies }\mathcal{F}\right)
\rightarrow
\begin{cases}
0 & ~ \text{if }p\leq (1+o(1))p^{\ast}\\
1 & ~ \text{if }p\geq (1+o(1))p^{\ast}
\end{cases}
$$
holds. One such example will be given in the following paragraph.
 
\medskip

{\bf Maker-Breaker connectivity game.}
The Maker-Breaker connectivity game is a game variant played on the edges of a graph $G$ with $\mathcal F$ consisting of all spanning trees of $G$. Lehman~\cite{L1964} stated that Maker wins the $(1:1)$~Maker-Breaker version of this game as the second player if and only if the graph $G$ contains two edge-disjoint spanning trees. Since the complete graph $K_n$ can be decomposed into even more spanning trees, a natural question is to ask what happens when Breaker's power gets increased by making his bias larger. 
Chv\'atal and Erd\H{o}s~\cite{CE1978} initiated the study of the $(1:b)$ variant and they could prove that its threshold bias is bounded from above by $(1+o(1))n/\ln n$. 
A matching lower bound was later given by Gebauer and Szab\'o \cite{GS2009}.

Now, if in the $(m:b)$ game on $K_n$ Maker and Breaker do not play according to a deterministic strategy but instead they play purely at random, the final graph consisting of Maker's edges will behave similarly to a random graph $G\sim G_{n,p}$ with $p=m/(m+b)$.
It is well known that the (sharp) threshold probability $p^\ast$
for $G\sim G_{n,p}$ being connected, i.e. where $G\sim G_{n,p}$ turns from almost surely being disconnected
to almost surely being connected, satisfies 
$p^\ast=(1+o(1))\ln n/n$ (see e.g.~\cite{B1985}, \cite{JLR2000}).
Surprisingly, when $m=1$, the latter corresponds to $b=(1+o(1))n/\ln n$ and thus perfectly matches the threshold bias mentioned above. In other words, for most values of $b$, a randomly played $(1:b)$ Maker-Breaker connectivity game on $K_n$ is very likely to end up with the same winner as the corresponding deterministically played game. This phenomenon usually is referred to as {\em probabilistic intuition}. There is a wide range of other games fulfilling this property as well, for example the
perfect matching game, the Hamiltonicity game \cite{K2011}
and the
doubly biased $(m:b)$ connectivity game when Maker's bias satisfies $m=o(\ln n)$~\cite{HMS2012}.
But there also exist games, where this intuition fails, such as the diameter game \cite{BMP2016} and the $H$-game \cite{BL2000}.\\

A different approach to give Breaker more power is to play unbiased, but to thin the board instead. Stojakovi\'c and Szab\'o \cite{SS2005} initiated the study of Maker-Breaker games played on a random graph $G \sim G_{n,p}$, their main question being to find the threshold probability $p^\ast$
at which an almost sure Breaker's win turns into an almost sure Maker's wins. The existence of such a (not necessarily sharp) threshold is guaranteed by the fact that the property of Maker having a winning strategy is monotone increasing. Now, for the connectivity game it is obvious that the threshold probability needs to satisfy 
$p^\ast\geq (1+o(1))\ln n/n$ since for smaller $p$ a random graph 
$G\sim G_{n,p}$ almost surely contains isolated vertices (see e.g.~\cite{B1985}, \cite{JLR2000}).
Stojakovi\'c and Szab\'o could show that
indeed $p\geq (1+o(1))\ln n/n$ is enough for 
Maker to win 
the connectivity game on $G\sim G_{n,p}$ almost surely.
Interestingly, this threshold probability
asymptotically equals the reciprocal of the threshold bias for
the corresponding Maker-Breaker game on $K_n$ --~another phenomenon which has also been observed for 
many other natural games 
(see e.g. \cite{FGKN2015},~\cite{HKSS2009}, \cite{NSS2016},~\cite{SS2005}).\\

{\bf Connector-Breaker games.} Recently, under the name PrimMaker-Breaker games,
London and Pluh\'ar~\cite{LP2018}
introduced a connected version of the Maker-Breaker games discussed above.
These games, which we will call {\em Connector-Breaker games} in the following, 
are played in the same way as the already described Maker-Breaker games,
with the only difference that {\em Connector}
(in the role of Maker) 
needs to choose her edges in such a way that her graph stays connected throughout the game. 
While London and Pluh\'ar \cite{LP2018}
studied the Connector-Breaker connectivity game on $K_n$, where
Connector aims for a spanning tree of $K_n$,
even more recently
Corsten, Mond, Pokrovskiy, Spiegel and Szab\'o~\cite{CMPSS2019} discussed the variant 
in which Connector aims for an odd cycle of $K_n$.
For the unbiased game, London and Pluh\'ar proved the following:

\begin{thm}
  Playing the $(1:1)$ Connector-Breaker game on a graph $G$ with $n$ vertices, Connector wins as the first player if and only if $G$ contains a copy of $H_n$, where $H_n$ is the graph $K_{n-2,2}$ with an additional edge inside its two-element color class.
\end{thm}

Moreover, one can easily see that for $b\geq 2$
the $(1:b)$ Connector-Breaker connectivity game is won by Breaker on every graph $G$~\cite{LP2018}. Thus, the threshold bias for such a game equals 2.
Also, if the game is played on $G\sim G_{n,p}$, then, by the theorem above,
$p$ needs to be almost 1 for Connector to have a winning strategy on $G$ almost surely. Note that both these observations are in huge contrast to the results for the Maker-Breaker analogue.
However, by increasing Maker's bias by just 1, London and Pluh\'ar~\cite{LP2018} showed that the situation changes suddenly.

\begin{thm}
Playing the $(2:b)$ Connector-Breaker game on $K_n$, Connector wins if ${b<n/(8 \ln n)}$, and Breaker wins if $b>n/\ln n$.
\end{thm}

This results shows that increasing Connector's bias makes a huge difference. In particular, the threshold bias in the $(2:b)$ variant is of the same order as in the corresponding Maker-Breaker game and thus, in contrast to the $(1:b)$ games, both variants behave similarly. Naturally, this made London and Pluh\'ar \cite{LP2018} ask whether something similar could be observed when playing the $(2:2)$ game on a graph $G\sim G_{n,p}$ and if it might behave similarly to the $(1:1)$ Maker-Breaker version. In this paper we show that the latter is not the case, and we prove the following result:

\begin{thm}\label{main}
The threshold probability $p^\ast$ for the $(2:2)$ Connector-Breaker connectivity game on $G\sim G_{n,p}$ is of size $n^{-2/3+o(1)}.$
\end{thm}

Hence, even if Connector's bias gets increased, a much denser random graph is necessary for Connector to have a chance at winning almost surely the connectivity game than in the respective Maker-Breaker variant of this game.
Since the proof of our theorem is rather technical and the proofs of the upper and lower bound require different techniques, we split the theorem into two parts.

\begin{thm}\label{main_breaker}
  Let $\eps >0$ be a constant. For $p\leq n^{-2/3-\eps}$ a random graph $G\sim G_{n,p}$ a.a.s.~has the following property:
  Playing a $(2:2)$ Connector-Breaker game
  on the edge set of $G$, Breaker has a strategy to keep a vertex isolated in Connector's graph.
\end{thm}

\begin{thm}\label{main_connector}
  Let $\eps >0$ be a constant. For $p\geq n^{-2/3+\eps}$ a random graph $G\sim G_{n,p}$ a.a.s.~has the following property:
  Playing a $(2:2)$ Connector-Breaker game
  on the edge set of $G$, Connector has a strategy to claim a
  spanning tree.
\end{thm}

\medskip

\subsection{Organization of the paper}~
The main focus of this paper is proving Theorem~\ref{main_breaker} and Theorem~\ref{main_connector}. In Section \ref{sec:Preliminaries} we will give an overview over all required tools. In the Sections~\ref{sec:Breaker's strategy} and~\ref{sec:Connector's strategy} we will describe Breaker's and Connector's strategy respectively. We will also state some lemmas, from which it will follow that the given strategies succeed almost surely for the respective ranges of the edge probability $p$. We postpone the proofs of these lemmas to Section~\ref{sec:algorithm} (for Breaker's strategy) and Section~\ref{sec:structures} (for Connector's strategy). Finally, we will give some concluding remarks in Section \ref{sec:concluding}.


\subsection{Notation and terminology}~
The game-theoretic and graph-theoretic notation in our paper is rather standard and most of the times it follows the notation of \cite{HKSS2014} and \cite{W2001}.

For a positive integer $n$, we set $[n]:=\{k\in\mathbb{N}:~ 1\leq k\leq n\}$.
For a graph $G=(V,E)$ we write $V(G)$ and $E(G)$
for the vertex set and the edge set of $G$, respectively.
If $\{v,w\}$ is an edge from $E(G)$, we denote it with $vw$ for short.
A vertex $w$ is called a neighbour of $v$ in $G$
if $vw\in E(G)$ holds.
The neighbourhood of $v$ in $G$ is
$N_G(v)=\{w\in V(G):~ vw\in E\}$,
and with $d_G(v)=|N_G(v)|$ we denote the degree of $v$ in $G$.
Let subsets $A,B\subset V(G)$ be given. We let
$N_G(v,A)=N_G(v)\cap A$ be the 
neighbourhood of $v$ in $A$, and we set
$d_G(v,A)=|N_G(v,A)|$ to be the degree of $v$ into $A$. 
Moreover, we let
$N_G(A):=\bigcup_{v\in A} N_G(v)$,
$e_G(A):=\{vw\in E(G):~ v,w\in A\}$
and
$e_G(A,B):=\{vw\in E(G):~ v\in A,w\in B\}$.

Let two graphs $H$ and $G$ be given.
If $V(H)\subset V(G)$ and $E(H)\subset E(G)$ holds, 
we call $H$ a subgraph of $G$,
and we write $H\subset G$ for short.
We also let $G\setminus H=(V(G),E(G)\setminus E(H))$ in this case.
If there is a bijection $f:V(H)\rightarrow V(G)$
such that $vw\in E(H)$ holds if and only if
$f(v)f(w)\in E(G)$ holds, the two graphs $H$ and $G$ are called isomorphic (denoted with $H\cong G$), and we also say that $H$ is a copy of $G$ in this case. 

A path $P$ with $V(P)=\{v_1,v_2,\ldots,v_k\}$ and
$E(P)=\{v_iv_{i+1}:~ 1\leq i\leq k-1\}$
will be represented by its sequence of
vertices, e.g. $P=(v_1,v_2,\ldots,v_k)$.
Its length is its number of edges.

Assume that some Connector-Breaker game, played on the edge set of some graph $G$, is in progress. At any moment during the game, let $C$ be the graph consisting of Connector's edges and let $B$ be the graph consisting of Breaker's edges. For short,
also set $V_C=V(C)$, $E_C=E(C)$ and $E_B=E(B)$.
If an edge belongs to $B\cup C$, we call it claimed;
otherwise it is called free.

Given a distribution $\mathcal{D}$ and a random variable
$X$, we write $X\sim \mathcal{D}$ for $X$ being sampled according to the distribution $\mathcal{D}$.
With $Bin(n,p)$ we denote the binomial distribution
with parameters $n$ and $p$.
Moreover, with $G_{n,p}$ we denote 
the Erd\H{o}s-Renyi random graph model on $n$ vertices and with edge probability $p$.
If $X$ is a random variable, we let $\Exp(X)$ denote its expectation. If $\E$ is an event, we let $\Prob(\E)$ denote its probability.
A sequence of events $\E_n$ is said to hold
asymptotically almost surely (a.a.s.) if 
$\Prob(\E_n)\rightarrow 1$ for $n\rightarrow \infty$.

Our main results are asymptotic.
Whenever necessary, we will assume $n$ to be large enough.
We will not optimize constants,
and whenever these are not crucial, we
will omit rounding signs.

%
%

\section{Preliminaries}\label{sec:Preliminaries}


\subsection{Maker-Breaker Box game} ~
A simple, yet very useful positional game is the following one, introduced by Chv\'atal and Erd\H{o}s~\cite{CE1978}, which usually is 
helpful to describe strategies that aim to bound the degrees in the opponent's graph. 
The game $Box(p,1;a_1,\ldots,a_n)$
is played on a hypergraph $(X,\mathcal{H})$,
with $\mathcal{H}=\left\{F_1,\ldots,F_n\right\}$
consisting of $n$ pairwise disjoint hyperedges (called {\em boxes}),
satisfying $|F_i|=a_i$ for every $i\in [n]$.
In every round, BoxMaker claims at most $p$ elements from 
$X$ that have not been claimed before, 
while BoxBreaker solely claims one such element.
If, throughout the game, BoxMaker succeeds in claiming all the elements of a box $F_i$, she is declared the winner of the game. Otherwise, i.e. when BoxBreaker succeeds in claiming at least one element in each box, BoxBreaker wins.
The following lemma is a well-known criterion for BoxBreaker to have a winning strategy in the Box game (see e.g.~\cite{CE1978}, \cite{HKSS2014}).

\begin{lemma}\label{lem:boxgame}
Let $a_i=m$ for every $i\in [n]$ and assume that $m>p(\ln n + 1)$, then BoxBreaker wins the game $Box(p,1;a_1,\ldots,a_n)$.
A winning strategy $\mathcal{S}$ is the following one: in every round, BoxBreaker claims an element which belongs to a box that he does not have an element from and which, among all such boxes, contains the largest number of Maker's elements. 
\end{lemma}

In fact, the first sentence in the above lemma is Theorem 3.4.1 in~\cite{HKSS2014}, while the mentioned strategy is contained in its proof. As an immediate corollary of the above lemma we obtain the following:

\begin{cor}\label{cor:boxgame}
Let BoxMaker and BoxBreaker play the game 
$Box(p,1;a_1,\ldots,a_n)$ with boxes $F_i$ of size 
$|F_i|=a_i\geq m$. Then following the strategy $\mathcal{S}$
from Lemma~\ref{lem:boxgame}, BoxBreaker can guarantee that the following holds for every $i\in[n]$ throughout the game: as long as he does not claim an element in $F_i$,
the number of BoxMaker's elements in $F_i$ is bounded
by $p(\ln n + 1)$.  
\end{cor}


\subsection{Probabilistic tools and basic properties of $G_{n,p}$} 
In this section we present a few bounds on large deviations
of random variables that will be used to identify typical edge
distributions in a random graph $G\sim G_{n,p}$. 
Most of the time, we will use the following inequalities due to Chernoff (see e.g.~\cite{AS2008}, \cite{JLR2000}).

\begin{lemma}\label{lem:Chernoff1}
If $X \sim Bin(n,p)$, then
\begin{itemize}
    \item $\Prob(X<(1-\delta)np)< \exp\left(-\frac{\delta^2np}{2}\right)$ for every $\delta>0$ , and
    \item $\Prob(X>(1+\delta)np)< \exp\left(-\frac{np}{3}\right)$ for every $\delta\geq 1$.
\end{itemize}
\end{lemma}

\begin{lemma}\label{lem:Chernoff2}
Let $X \sim Bin(n,p)$ with expectation $\mu=\Exp(X)$, and 
let $k \geq 7\mu$, then
$$\Prob(X \geq k) \leq e^{-k}.$$
\end{lemma}

Moreover, we will make use of the well-known Markov inequality (see e.g.~\cite{JLR2000}).

\begin{lemma}\label{lem:markov}
Let $X\geq 0$ be a random variable. For every $t\geq 0$ it holds that
$$
\Prob\left( X\geq t \right)
\leq \frac{\Exp(X)}{t}~ .
$$
\end{lemma}

As a first application of Chernoff's inequalities
we will prove a few simple bounds on degrees 
that are very likely to hold in a random graph $G\sim G_{n,p}$.

\begin{lemma}\label{lem:degree_bound}
Let $\eps>0$, $p=n^{-2/3-\eps}$ and let $G\sim G_{n,p}$.
Then with probability at least ${1-\exp(-n^{1/3-2\eps})}$
every vertex $v\in V(G)$ satisfies
\begin{equation}\label{cl:deg}
d_G(v)<2n^{\frac13-\eps}~ .
\end{equation}
\end{lemma}

\begin{proof}
For $v\in V(G)$ we have $d_G(v)\sim Bin(n-1,p)$
with $\Exp(d_G(v))=(n-1)p\sim np$. Applying 
Lemma~\ref{lem:Chernoff1} we deduce that
$
\Prob\left(d_G(v)>2np\right) \leq 
	\exp\left( - \frac14 n^{1/3-\eps}\right).
$
Taking a union bound over all possible vertices $v$,
the claim follows.
\end{proof}

\begin{lemma}\label{lem:tec3_connector}
Let $\eps >0$, $p=n^{-2/3+\eps}$ and let $G\sim G_{n,p}$.
Let $A\subset V(G)$ be of size $n^{2/3}$,
then with probability at least 
$1-\exp(-n^{\eps/2})$ 
every vertex $v\in V(G)\setminus A$ satisfies
$d_G(v,A)>n^{\eps/2}$.
\end{lemma}

\begin{proof}
Let $A$ be a fixed set of size $n^{2/3}$. 
Generating $G\sim G_{n,p}$ yields that
for every vertex $v\in V(G)\setminus A$, we have
$d_G(v,A) \sim Bin(|A|,p)$ and thus 
$\Exp(d_G(v,A)) = n^{\eps}.$
By Lemma~\ref{lem:Chernoff1} 
we deduce that 
$
\Prob\left(
d_G(v,A)\leq \frac{1}{2} n^{\eps}
\right) <  \exp\left( - \frac{1}{8} n^{\eps} \right)
$. Taking a union bound over all possible $v$, the claim follows.
\end{proof}

%
%

\section{Breaker's strategy} \label{sec:Breaker's strategy}
\subsection{Defining bad vertices}

For $p=n^{-2/3-\eps}$ we aim to give a Breaker's strategy
that a.a.s. isolates a given vertex $x$ from Connector's graph
when a $(2:2)$ game is played on $G\sim G_{n,p}$.
In order to do so, we first define iteratively a set $B^x$
of vertices that are {\em bad} with respect to the aim of isolating $x$. If $x$ is carefully chosen (which we will manage later) then Breaker has a strategy to make sure that Maker in her move either does not even reach $B^x$, or in case she reaches $B^x$ then Breaker can immediately destroy all potential threads. More details will be given later. Algorithm~\ref{alg:bad} decribes how $B^x$ is constructed.

  \begin{algorithm}[ht]\label{alg:bad}
    \caption{Bad vertex set $B^x$ for given vertex $x$}    
\SetKwInOut{Input}{Input}
\SetKwInOut{Output}{Output}
\Input{~ graph $G$ and vertex $x\in V(G)$}
\Output{~ number of iterations $r_x$, bad vertex set 
$B^x=\bigcup_{k\leq r_x} B_k^x$}
  ~ $B^x_1:=N_G(x)$\;
  ~ $B^x:=B^x_1$\;
    \For{$i\geq 2$}{
    \quad\mbox{} $B^x_i:=\left\{v\notin B^x\cup \{x\}:~ 			d_G(v,B^x)\geq 2\right\}$ \;
	\quad\mbox{} $B^x\leftarrow B^x\cup B^x_i$\\
         \lIf{$B^x_{i}=\varnothing$\\ 
         \quad\mbox{}}{
        halt with output $B_1^x,\ldots,B_{i-1}^x,B^x$ and 		
        $r_x=i-1$}}
  \end{algorithm}

The following lemma will be crucial for Breaker's strategy.

\begin{lemma}\label{lem:bad}
Let $n$ be a large enough integer and
let $\eps \geq 7 \ln\ln n / \ln n$.
For $p=n^{-2/3-\eps}$ generate $G\sim G_{n,p}$.
Then a.a.s. $G$ satisfies the following property:
For every set $M\subset V(G)$ of size~3, there exists a vertex $x$
such that Algorithm~\ref{alg:bad} produces 
a set $B^x$ of vertices and a sequence
$(B_1^x,\ldots,B_{r_x}^x)$ of disjoint subsets of $B^x$
such that the following holds:
{\begin{enumerate}[label=\itmarab{B}]
\item\label{Bprop:neighbours1} $B_1^x=N_G(x)$
and $e_G(B_1^x)=0$,
\item\label{Bprop:neighbours2} for every $2\leq i\leq r_x$ and every vertex in $v\in B_i^x$ we have 
$d_G\left(v,\bigcup_{k\leq i} B_k^x\right)=2$,
\item\label{Bprop:nonbad} for every vertex $v\in V\setminus (B^x\cup \{x\})$ it holds that $d_G(v,B^x)\leq 1$,
\item\label{Bprop:disjointM} $B^x\cap (M\cup N_G(M))  = \varnothing$.
\end{enumerate}}
\end{lemma}

We postpone the proof of the above lemma to Section~\ref{sec:algorithm} and recommend to read Breaker's strategy
first.

\subsection{The strategy}
In the following we prove Theorem~\ref{main_breaker}. Let Connector and Breaker play a $(2:2)$ game on $G\sim G_{n,p}$.
We will show that, under the condition that the property described in Lemma~\ref{lem:bad} holds,
Breaker has a strategy that isolates a vertex from Connector's graph.
Let $V_C^r$ denote the set of vertices that are covered by Connector's edges at the end of round $r$. Immediately after Connector's first move, we have 
$|V_C^1|=3$ and thus, by the property from Lemma~\ref{lem:bad} (applied with $M=V_C^1$), we find a vertex $x$ such that
Algorithm~\ref{alg:bad} produces 
a set $B^x$ of vertices and a sequence
$(B_1^x,\ldots,B_{r_x}^x)$ of disjoint subsets of $B^x$ such that the Properties~\ref{Bprop:neighbours1}--\ref{Bprop:disjointM}
hold with $M=V_C^1$. 
Notice that, at this point $x\notin V_C^1\cup N_G(V_C^1)$
holds, according to \ref{Bprop:disjointM} and since $N_G(x)\subset B^x$.

In order to simplify notation,
let $B_0^x:=\{x\}$ and set
$B_{<i}^x := \bigcup_{\ell=0}^{i-1} B_{\ell}^x$
as well as $B_{\leq i}^x := \bigcup_{\ell=0}^{i} B_{\ell}^x$.
Breaker's strategy is to make 
sure that for each round $r$, immediately 
after his move the following property holds for every free edge $vw$:

\medskip\medskip

\begin{center}
\begin{minipage}{0.8\textwidth}
{\begin{enumerate}[label=\itmarab{Q}]
\item\label{B_maintain}
If there exists $0\leq i\leq r_x$ 
such that $v\in (N_G(V_C^r)\setminus V_C^r) \cap B_i^x$
and $w\in V_C^r$, then
$w\in B_{<i}^x$.
\end{enumerate}}
\end{minipage}
\end{center}

\medskip\medskip

Let us observe first that Breaker keeps $x$ isolated in Connector's graph, if he is indeed able to maintain \ref{B_maintain} for every free edge after each of his moves.
Assume this is not the case, i.e. there is some round $r$ in which Connector reaches vertex $x$.
Then immediately after Breaker's $(r-1)^\text{st}$
move, we have that \ref{B_maintain} holds 
for every free edge and still
$x\notin V_C^{r-1}$.
From this it follows that immediately before Connector's $r^{\text{th}}$ move there cannot be a free edge $xw$ with $w\in V_C^{r-1}$. 
Indeed, otherwise we would need 
$x\in (N_G(V_C^{r-1})\setminus V_C^{r-1})\cap B_0^x$ 
and by \ref{B_maintain} we would get 
$w\in B_{<0}^x = \varnothing$, a contradiction.
Thus, in order to reach $x$ during round $r$, Connector would need do claim a path $(w,v,x)$
of length 2, starting with some vertex $w\in V_C^{r-1}$ and ending in $x$. It then follows that 
$v\in (N_G(V_C^{r-1})\setminus V_C^{r-1}) \cap B_1^x$.
However, using \ref{B_maintain} 
for the free edge $wv$ at the end of round~$r-1$, this 
would give
$w\in B_{<1}^x = B_0^x$ and thus $x=w$,
a contradiction. 

Hence, we know that Connector cannot reach $x$ as long as Breaker restores \ref{B_maintain}
for every free edge.
It thus remains to verify that Breaker can indeed do so. We proceed by induction. \\

For round 1, observe that immediately after Connector's first move, there is no edge between $V_C^1$ and $B^x\cup \{x\}$, according to Property~\ref{Bprop:disjointM} (with $M=V_C^1$).
Thus, Property~\ref{B_maintain} holds at the end of round 1 for every free edge, independent of what Breaker's first move is,
as there does not exist any edge $vw$ as described in that property.\\
Let us assume then, that \ref{B_maintain} is satisfied immediately after Breaker's 
$(r-1)^{\text{st}}$
move for every free edge, and let us explain how Breaker restores
\ref{B_maintain} in the next round. 
Without loss of generality we may assume that in round $r$ Connector reaches exactly two new vertices, say $w_1$ and $w_2$, i.e. $V_C^{r}=V_{C}^{r-1}\cup \{w_1,w_2\}$.

If after Connector's $r^{\text{th}}$ move,
there exist at most two free edges that fail to satisfy
Property~\ref{B_maintain} (with $V_C=V_C^r$),
then Breaker claims these edges and by this easily restores
that \ref{B_maintain} holds for every free edge
at the end of round $r$.
So, assume for a contradiction that immediately after Connector's $r^{\text{th}}$ move there are 
at least three free edges that do not satisfy \ref{B_maintain}. All of these edges need to be incident to $w_1$ or $w_2$, as before Connector's move the Property~\ref{B_maintain}
was true for every free edge (where $V_C=V_C^{r-1}$). Without loss of generality let $w_2$
be incident to at least two of these edges,
say $w_2v_1$ and $w_2v_2$. As these edges fail to hold \ref{B_maintain} after Connector's $r^{\text{th}}$ move, we have
$v_1\in (N_G(V_C^r)\setminus V_C^r) \cap B_{i_1}^x$
and 
$v_2\in (N_G(V_C^r)\setminus V_C^r) \cap B_{i_2}^x$
for some $0\leq i_1,i_2\leq r_{x}$,
while $w_2\in V_C^r$ and 
$w_2\notin B_{<i}^x$ with $i:=\max\{i_1,i_2\}$.
Now, since $w_2$ has two neighbours in $B^x\cup \{x\}$,
Algorithm~\ref{alg:bad} at some point must have added $w_2$ to $B^x$. Thus, we conclude that $w_2\in B_k^x$
for some $k \geq \max\{i_1,i_2\}$.
 
Consider first the case that in round $r$
Connector reaches $w_2$ by claiming a free edge
$yw_2$ with $y\in V_C^{r-1}$.
Then $y\notin \{v_1,v_2\}$. Moreover,
$w_2\in N_G(V_C^{r-1})\setminus V_C^{r-1}$
and, since \ref{B_maintain}
was true for $yw_2$ at the end of round~$r-1$
(with $V_C=V_C^{r-1}$),
we conclude $y\in B_{<k}^x$. But this means that
$w_2\in B_k^{x}$ has three neighbours in
$B_{\leq k}^x$ (namely $v_1,v_2$ and $y$), a contradiction to~\ref{Bprop:neighbours2}.

Consider then the case that in round $r$
Connector does not reach $w_2$ as in the first case.
That is, in round $r$ Connector 
claims a path $(y,w_1,w_2)$ with $y\in V_C^{r-1}$
and
$w_1\in N_G(V_C^{r-1})\setminus V_C^{r-1}$. 
We know that $w_2\in B_k$ has exactly two neighbours
in $B_{\leq k}^x$ according to Property~\ref{Bprop:neighbours2}, and these neighbours need to be
$v_1$ and $v_2$. It follows that the third edge, which does not satisfy~\ref{B_maintain} immediately before Breaker's $r^{\text{th}}$ move, cannot be incident to $w_2$ and thus needs to be of the form $v_3w_1$ with $v_3\in (N_G(V_C^r)\setminus V_C^r)\cap B_{i_3}^x$ for some $0\leq i_3\leq r_x$.
Then $v_3,w_2\in B^x$ are two neighbours of $w_1$
and hence Algorithm~\ref{alg:bad} must have added $w_1$ to $B^x$ at some point, say $w_1\in B_t^x$. 
Since again $w_2\in B_k^x$ has exactly two neighbours in $B_{\leq k}^x$ and these are
$v_1$ and $v_2$, we must have $w_1\notin B_{\leq k}^x$, i.e. $t>k$ .
But now, by induction, Property~\ref{B_maintain}
was true for the free edge $yw_1$ at the end of round $r-1$, and thus $y\in B_{< t}^x$. 
Moreover, as we assumed $v_3w_1$ to be an edge not satisfying \ref{B_maintain} after Connector's
$r^{\text{th}}$ move, we have $w_1\notin B_{<i_3}^x$ and thus $i_3\leq t$.
Hence, we obtain that the three neighbours 
$w_2,v_3,y$ 
of $w_1\in B_t^x$ belong to $B_{\leq t}^x$, 
as we have $w_2\in B_k^x\subset B_{\leq t}^x$ and
$v_3\in B_{i_3}^x\subset B_{\leq t}^x $
and $y\in B_{<t}^x$. This again leads to a contradiction with \ref{Bprop:neighbours2}.
\hfill $\Box$

%
%

\section{Connector's strategy} \label{sec:Connector's strategy}

\subsection{Defining good structures}
For $p=n^{-2/3+\eps}$ we aim to give a Connector's strategy with which Connector a.a.s. can reach every vertex of $G\sim G_{n,p}$.
In order to do so, we will first describe a few useful structures, that are typically contained in $G$ even after deleting a few edges 
and which will help
Connector later on to reach any fixed vertex within
a small number of rounds.

Recall that $E_B$ denotes the set of Breaker's edges at any moment during a Connector-Breaker game,
while $V_C$ denotes the set of vertices incident to Connector's edges. Moreover,
denote with $\Tk$ the full binary tree
with $k$ levels.

\begin{dfn}\label{def:good}
Let $k\in\mathbb{N}$. Assume a $(2:2)$ Connector-Breaker game 
on some graph $H$ is in progress.
Let $x\in V(H)\setminus V_C$. Then we call a subgraph $T \cong \Tk$
of $H$ \textit{good} with respect to 
$(x,H)$
if the following conditions hold:
\begin{enumerate}
\item[(1)] $x\notin V(T)$,
\item[(2)] if $\overrightarrow{T}$ is the orientation
where the edges are oriented from the root to the leaves, then for every arc $\overrightarrow{uw}\in E(\overrightarrow{T})$ we have either $uw\notin E_B$ or ($uw\in E_B$ and $w\in V_C$).
\item[(3)] for every leaf $v$ of $T$ we have $vx\in E(H)\setminus E_B$. 
\end{enumerate}
\end{dfn}

\begin{lemma}[Base strategy]\label{lem:connector:basic}
Assume a $(2:2)$ Connector-Breaker game 
on some graph $H$ is in progress with
Connector being the next player to make a move. 
Let $x\in V(H)\setminus V_C$
and let $k\geq 2$ be any integer.
Moreover, assume that $H$ contains a binary tree $T \cong \Tk$
which is good with respect to $(x,H)$
and such that its root $r$ belongs to 
$V_C$ already.
Then Connector has a strategy $\mathcal{S}_x$ 
to
reach $x$ (i.e. to add $x$ to $V_C$) 
within at most $k$ rounds.
\end{lemma}

\begin{proof}
We prove the statement by induction on $k$.
For $k=2$, by assumption we are given 
a tree $T\cong \mathcal{T}_2$
the leaves of which are adjacent with $x$ in $H \setminus B$, according to Definition~\ref{def:good}. If one of the leaves belongs to $V_C$, then Connector can take the edge between that leaf and $x$. 
Otherwise, according to (2) we obtain
$E(T)\cap E_B=\varnothing$. Then, since the root $r$ of $T$ belongs to $V_C$ by assumption, Connector can claim one edge
between $r$ and a leaf of $T$, and 
for her second edge she can claim the edge between that leaf and $x$. Thus, she reaches $x$ within $1$ round.

Let $k>2$ then. Let $T\cong \Tk$ be a tree as described in the assumption of the lemma. Denote the root of $T$
with $r$, let $r_1$ and $r_2$ be the neighbours of $r$ in $T$,
and let $r_{1,1}$, $r_{1,2}$ and $r_{2,1}$, $r_{2,2}$ be the respective children of $r_1$ and $r_2$ in $T$. Each of the vertices $r_{i,j}$ is the root of a subtree $T_{i,j}\cong \mathcal{T}_{k-2}$
the leaves of which are adjacent with $x$ in $H\setminus B$. Let
$$
E_{i,j}:=\{r_ir_{i,j}\}\cup E(T_{i,j})\cup \left\{
xw:~ w~\text{is a leaf of }T_{i,j}\right\}
$$
for every $1\leq i,j\leq 2$, 
and observe that the four sets $E_{i,j}$ are pairwise disjoint. 
For the first round, Connector makes sure that
$r_1$ and $r_2$ are added to $V_C$
if they do not belong to $V_C$ already. This is possible since
for every $i\in [2]$ we have that $r_i\in V_C$ already before that round or $r_ir\notin E_B$ according to (2) in  Definition~\ref{def:good}.
After Breaker's following move we know that
there are at least two sets $E_{i,j}$
with $1\leq i,j,\leq 2$ that Breaker did not touch in his move. Taking the union of two such sets,
say $E_{i_1,j_1}$ and $E_{i_2,j_2}$,
while identifying $r_1$ with $r_2$ if
$i_1\neq i_2$, we obtain
a binary tree $T'\cong \mathcal{T}_{k-1}$
which is good with respect to $(x,H')$
where $E(H')=E_{i_1,j_1}\cup E_{i_2,j_2}$. 
Thus, by induction Connector needs at most $k-1$ further rounds for reaching $x$.
\end{proof}

Connector's main strategy will be split into different stages. Depending on the number of rounds played so far, she will use similar but different structures that help to increase $V_C$ until every vertex is reached. These structures are given by the following lemmas while the proofs of the lemmas will be given in Section~\ref{sec:structures}.

\begin{lemma}[Good structures for Stage I]\label{lem:tec1_connector}
For every constant $\delta >0$ 
there exists an integer $k_1\in\mathbb{N}$ such that the following holds.
Let $G\sim G_{n,p}$ with $p=n^{-2/3+\delta}$, 
then with probability at least 
$1-n^{-1}$ the following is true for every $r,x\in V(G)$:\\
Let $B$ be any subgraph of $G$ with $e(B)\leq n^{1/3}\ln n$, then the graph $G\setminus B$ contains a copy $T$ of $\mathcal{T}_{k_1}$ such that
$r$ is the root of $T$, $x\notin V(T)$ and
every leaf of $T$ is adjacent to $x$ in $G\setminus B$. 
\end{lemma}

\begin{lemma}[Good structures for Stage II]\label{lem:tec2_connector}
For every constant $\delta >0$ 
there exists an integer $k_2\in\mathbb{N}$ such that the following holds.
Let $G\sim G_{n,p}$ with $p=n^{-2/3+\delta}$,
and let $A\subset V(G)$ be of size $n^{1/3}$,
then with probability at least 
$1-n^{-1}$
the following is true for every $x\in V\setminus A$:

Let $M$ be any subset of $V\setminus \{x\}$,
let $B$ be any subgraph of $G$ with
$d_B(v) \leq \ln^2 n$ for every $v\in V\setminus M$ and such that $e(B)\leq n^{2/3}\ln n$, then
$G$ contains
a vertex $z\in N_{G\setminus B}(A)$ and four vertex disjoint copies $T_{\ell}$ of $\mathcal{T}_{k_2}$ with roots $r_{\ell}$
such that for every $\ell\in [4]$ we have: 
{\begin{enumerate}[label=\itmarab{S}]
\item\label{good2:x} $x\notin V(T_{\ell})$,
\item\label{good2:top} $zr_{\ell}\in E(G\setminus B)$,
\item\label{good2:edges} if $\overrightarrow{T_{\ell}}$ is the orientation
where the edges are oriented from the root to the leaves, then for every arc $\overrightarrow{uw}\in E(\overrightarrow{T_{\ell}})$ we have either $uw\notin E(B)$ or ($uw\in E(B)$ and $w\in M$),
\item\label{good2:leaves} for every leaf $v$ of $T_{\ell}$ we have $vx\in E(G\setminus B)$. 
\end{enumerate}}

\end{lemma}

\begin{figure}[h]
\begin{center}
\includegraphics[scale=0.5,page=2]{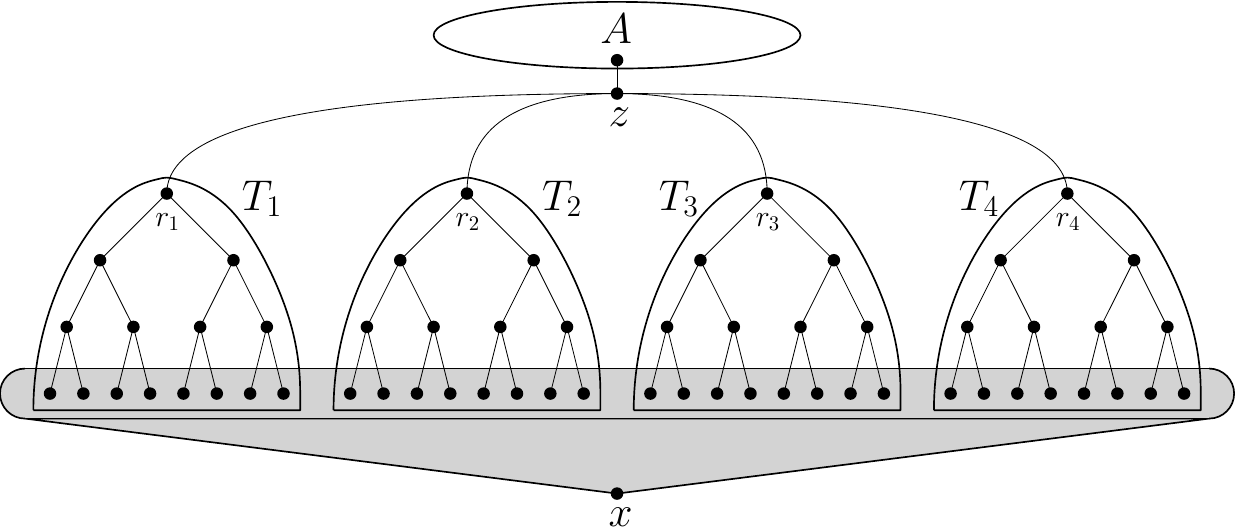}
\caption{Structure for Stage II}
\label{fig:structure2}
\end{center}
\end{figure}

We postpone the proofs of the above lemmas to Section~\ref{sec:structures} and recommend to read Connector's strategy
first.

\subsection{The strategy}

In the following we prove Theorem~\ref{main_connector}. 
Let $\eps>0$ be given, and let 
$k_1$ and $k_2$ be integers promised
by Lemma~\ref{lem:tec1_connector}
and Lemma~\ref{lem:tec2_connector} (applied with $\delta=\eps$),
respectively. Set $k:=\max\{k_1,k_2\}+2$.
Before revealing $G\sim G_{n,p}$ on the vertex set $V=[n]$, 
we fix an arbitrary set $A_1\subset [n]$ of size $n^{1/3}$ 
and an arbitrary set $A_2\subset [n]$ of size $n^{2/3}$.
Then, with probability tending to~1, all the properties
from 
Lemma~\ref{lem:tec1_connector}, Lemma~\ref{lem:tec2_connector} (applied for $A=A_1$)
and Lemma~\ref{lem:tec3_connector} (applied for $A=A_2$) hold.
From now on, let us condition on these.
Let Connector and Breaker play a $(2:2)$ game on $G$.
In the following we will first describe a strategy for Connector, and afterwards we will show that indeed it constitutes a winning strategy for the connectivity game on $G$, when we assume all the properties that we conditioned on above to hold.
The strategy will be described through the following two stages between which Connector alternates. If at any moment Connector cannot follow the strategy while $V\neq V_C$ still holds, then she forfeits the game. (We will show later that this does not happen). 

\medskip

{\bf Strategy description:} Fix a vertex $r\in V$ to be Connector's start vertex, and set $V_C=\{r\}$ before the game starts. As long as $V\neq V_C$ holds, Connector plays as follows, starting with Stage~I for her very first move.

\begin{itemize}
\item[] \textbf{Stage I:} Let $x\in V\setminus V_C$ 
be an arbitrary vertex, where we first prefer the
vertices of $A_1$, secondly prefer the vertices of 
$A_2$ and only afterwards consider all the remaining vertices. 
Connector then 
adds the vertex $x$ to $V_C$ within at most $k$ rounds.
The details of how she can do this
can be found later in the strategy discussion.
Immediately afterwards, if still $V\neq V_C$ holds, Connector proceeds
with Stage II.

\medskip

\item[] \textbf{Stage II:} 
Let $x\in V\setminus V_C$ be an arbitrary vertex
maximizing $d_B(x)$ among all vertices in $V\setminus V_C$.  
Connector then 
adds the vertex $x$ to $V_C$ within at most $k$ rounds.
The details of how she can do this
can be found later in the strategy discussion.
Immediately afterwards, if still $V\neq V_C$ holds, Connector proceeds
with Stage I.

\end{itemize}

\medskip

{\bf Strategy discussion:} If Connector can follow the strategy, without forfeiting the game, until $V=V_C$ holds,
then it is obvious that she succeeds in occupying a spanning tree and thus wins the game. It thus remains to prove that Connector always can follow the proposed strategy. 
In order to so, we start with two simple observations.

\begin{obs}\label{obs:degbound}
For as long as Connector can follow the proposed strategy, it holds that $d_B(v)<\ln^2 n$ for every $v\in V\setminus V_C$.
\end{obs}

\textit{Proof.} While the Connector-Breaker game on $G$ is going on, let us consider the Box game $Box(8k,1;n-1,\ldots,n-1)$
where for every vertex $i\in V(G)$ there is a box
$F_i$ of size $n-1$. In this auxiliary game, let Breaker take over the role of BoxMaker and let Connector be BoxBreaker in the following way. Whenever Breaker claims some edge $uw$ in the game on $G$, let BoxMaker claim one element in each of the boxes $F_u$ and $F_w$. Observe that this way, the number of BoxMaker's elements in any box $F_v$ will be equal to $d_B(v)$. 
Furthermore, whenever in Stage II Connector fixes some vertex $x$ of largest degree $d_B(x)$
(in order to add this vertex to $V_C$ within the following $k$ rounds),
let BoxBreaker claim an element in the box $F_x$.
Observe that everything is within the rules then,
as the latter always repeats within at most $2k$ rounds
in which BoxMaker may get up to $2k\cdot 4=8k$ new elements
over all the boxes.

Now, Corollary~\ref{cor:boxgame} ensures that 
whenever a vertex $x\in V\setminus V_C$ is selected for Stage II,
right at this moment we have 
\begin{equation*}
d_B(x)=|F_x|<8k(\ln n + 1).
\end{equation*}
Since such a vertex $x$ is always chosen to have maximal Breaker degree among all vertices in $V\setminus V_C$ and since such a choice always repeats within at most $2k$ rounds, we obtain
$$
d_B(v)<8k(\ln n + 1) + 2k\cdot 2 < \ln^2 n
$$
whenever $v\in V\setminus V_C$. This proves the observation.  \done

\begin{obs}\label{obs:roundsbound}
As long as Connector can follow the proposed strategy the following holds:
\begin{enumerate}
\item[(i)] If $A_1\not\subset V_C$, then $e(B)<n^{1/3} \ln n$.
\item[(ii)] If $A_2\not\subset V_C$, then $e(B)<n^{2/3} \ln n$.
\end{enumerate}
\end{obs}

\textit{Proof.} Stage~I always repeats after at most $2k$ rounds. Since for Stage I Connector prefers the vertices of $A_1$ to be added to $V_C$, it takes her at most $2k|A_1|$ rounds until $A_1\subset V_C$ holds, if she is able to follow the strategy. Thus, as long as $A_1\not\subset V_C$ holds, Breaker cannot have more than $2k|A_1|\cdot 2 < n^{1/3}\ln n$ edges. This proves statement (i). Statement (ii) can be proven analogously.~\done

\medskip

Now, using the Observations~\ref{obs:degbound} and \ref{obs:roundsbound} as well as the properties from Lemma~\ref{lem:tec1_connector},
\ref{lem:tec2_connector} and \ref{lem:tec3_connector}, we finally will show that Connector can always follow the proposed strategy. That is, assuming that so far Connector could follow her strategy, we will show that when she 
fixes her next vertex $x$ according to Stage~I or Stage~II, she can really add this vertex to $V_C$ within at most $k$ rounds.
In order to do so, we will consider three cases.

\medskip

{\bf Case 1 ($A_1\not\subset V_C$):}
In this case we have $e(B)\leq n^{1/3}\ln n$ according to Observation~\ref{obs:roundsbound}.
Thus, by the property from Lemma~\ref{lem:tec1_connector}
we can find a copy $T$ of $\mathcal{T}_{k_1}$ in $G\setminus B$
such that $r$ is the root of $T$, such that $x\notin V(T)$ and such that every leaf of $T$ is adjacent to $x$ in $G\setminus B$.
In particular, $T$ is good with respect to $(x,G\setminus B)$.
Thus, following the base strategy $\mathcal{S}_x$ from Lemma~\ref{lem:connector:basic}, Connector can reach $x$ within $k_1\leq k$ rounds.

\medskip

{\bf Case 2 ($A_1\subset V_C$ and $A_2\not\subset V_C$):}
In this case we have $e(B)\leq n^{2/3}\ln n$ according to Observation~\ref{obs:roundsbound},
and $d_B(v)<\ln^2 n$ for every $v\in V\setminus V_C \subset V\setminus A_1$ according to Observation~\ref{obs:degbound}. Applying the property from 
Lemma~\ref{lem:tec2_connector} (with $M=V_C$ and $A=A_1$)
we can find a vertex $z\in N_{G\setminus B}(A_1)$ 
and four vertex disjoint copies $T_{\ell}$ of $T_{k_2}$
with roots $r_{\ell}$ such that for every $\ell\in [4]$ we have that $zr_{\ell}\in E(G\setminus B)$ and $T_{\ell}$ is good with respect to $(x,G\setminus B)$. In the first round, Connector
claims an edge between $A_1$ and $z$ which is possible
as $A_1\subset V_C$ and $z\in N_{G\setminus B}(A_1)$.
Afterwards, consider the pairwise disjoint sets
$$
E_{\ell}:=\{zr_{\ell}\}\cup E(T_{\ell})\cup \left\{
xw:~ w~\text{is a leaf of }T_{\ell}\right\}
$$
for $\ell\in [4]$. As in the meantime Breaker claims only two edges, there will be at least two of these sets that Breaker does not touch until Connector's next move. Without loss of generality let these be the sets $E_1$ and $E_2$. Then
the union $\{zr_1,zr_2\}\cup E(T_1)\cup E(T_2)$ induces a copy of $\mathcal{T}_{k_2+1}$,
which is good with respect to $(x,G\setminus B)$.
Therefore, following the base strategy $\mathcal{S}_x$ from Lemma~\ref{lem:connector:basic}, Connector can reach $x$ within at most $k_2+1$ further rounds.
Hence, in total, Connector needs at most $k_2+2\leq k$ rounds in this case.

\medskip

{\bf Case 3 ($A_1\cup A_2\subset V_C$ and $V_C\neq V$):}
According to Observation~\ref{obs:degbound},
we have $d_B(x)<\ln^2 n$ before Connector
wants to add $x$ to $V_C$.
Following the property from Lemma~\ref{lem:tec3_connector} 
(with $A=A_2$) we then conclude that $d_G(x,A_2)>n^{\eps/2}>d_B(x)$.
Therefore, since $A_2\subset V_C$, Connector immediately can claim an edge leading to $x$. \hfill $\Box$

\section{Analysis of Algorithm~\ref{alg:bad}}\label{sec:algorithm}

The aim of this section is to prove Lemma~\ref{lem:bad}.
 For that reason we will prove a slightly more general lemma, Lemma~\ref{lem:tec}, from which Lemma~\ref{lem:bad} will follow.
For Lemma~\ref{lem:tec} we are going to apply Algorithm~\ref{alg:bad} to a set $A=\{x_1,\ldots,x_t\}$ of vertices, later choosing one of them carefully to obtain a vertex $x$ as promised by
Lemma~\ref{lem:bad}. 
That is, we first fix $x_1$ and apply Algorithm~\ref{alg:bad} in order to determine the set $B^{x_1}$, then we repeat the algorithm for $x_2$ and so on. Amongst other properties we will obtain that it
is very likely that all the sets $B^{x_j}$ are pairwise disjoint and satisfy certain degree conditions. To simplify notation we set
\begin{equation}\label{eq:Bji}
B^{(j,i)} := \bigcup_{\ell < j} B^{x_{\ell}} \cup \bigcup_{k\leq i} B_k^{x_j}~ .
\end{equation}

\begin{figure}[h]
\begin{center}
\includegraphics[scale=0.7]{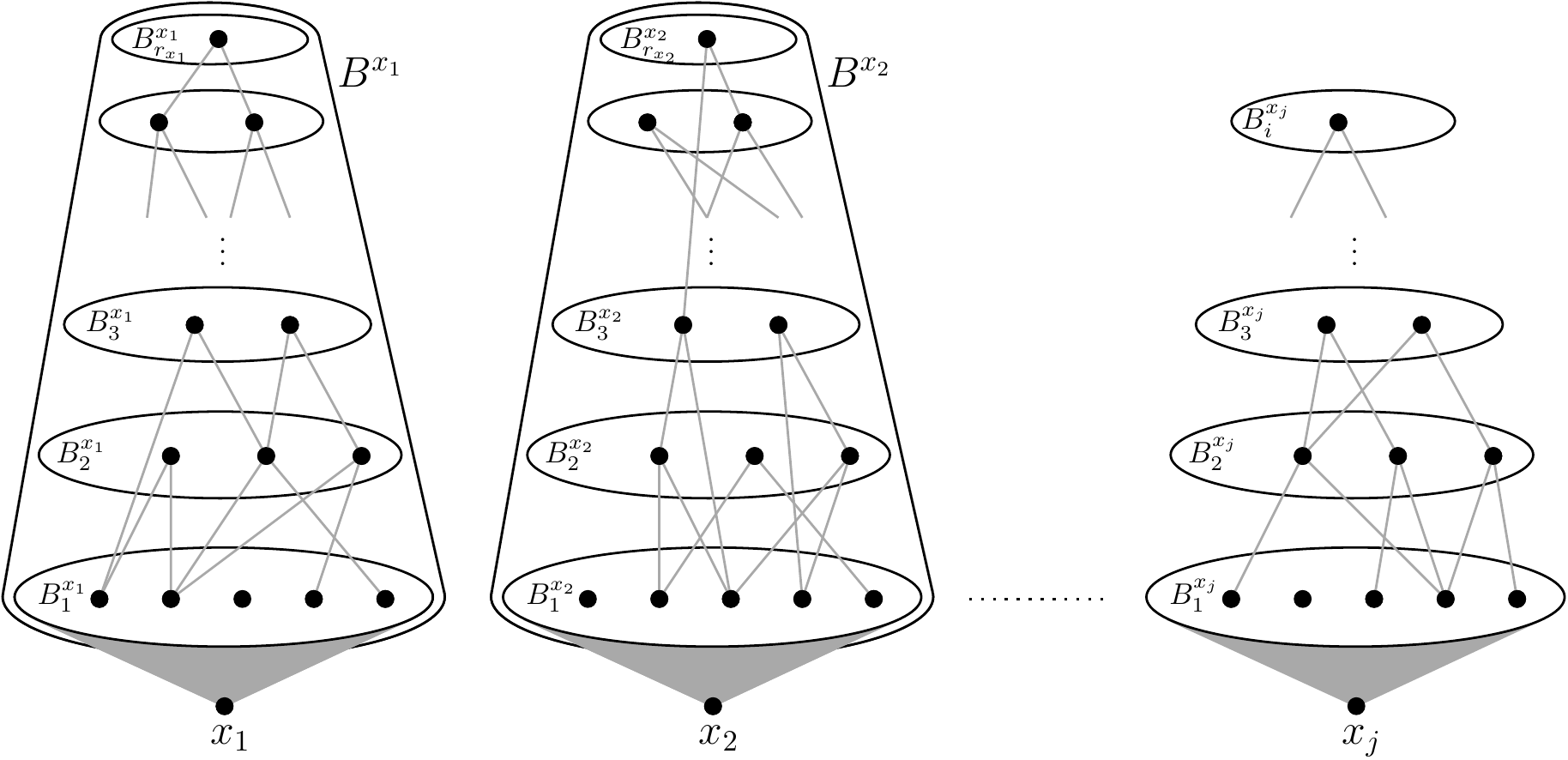}
\caption{Structure of $B^{(j,i)}$}
\label{fig:Bji}
\end{center}
\end{figure}

That is, $B^{(j,i)}$ is the set of all bad vertices that are determined immediately after $B_i^{x_j}$ is created. 
In particular, 
$B^{(t,r_{x_t})} = \bigcup_{x\in A} B^{x}$ is the union of all bad vertices after the algorithm is proceeded for
all vertices $x_j$.
Moreover, we let
$$
a(j,i):=
\begin{cases}
(j,i-1)~ , & i\neq 1\\
(j-1,r_{x_{j-1}})~, & i=1
\end{cases}
$$

\vspace{1mm}

denote the pair coming immediately before $(j,i)$ in lexicographic order, for $(j,i)\neq (1,1)$.

\newpage

\begin{lemma}[Technical Lemma]\label{lem:tec}
Let $n$ be a large enough integer,
let $\eps \geq 7 \ln\ln n / \ln n$
and let $t\in\mathbb{N}$ be any constant. 
For $p=n^{-2/3-\eps}$ generate a random graph $G\sim G_{n,p}$.
Then with probability at least $1-n^{-\eps/4}$ there
exists a set $A=\{x_1,\ldots,x_t\}\subset V(G)$
of size $t$, such that successively applying Algorithm~\ref{alg:bad} 
for $x_1,\ldots,x_t$ the following holds for every $j\in [t]$ and $i \leq \tilde{r}_j:= \min\{ r_{x_j}, \lceil 1/\eps \rceil\}$:

\begin{enumerate}[label=\itmarab{P}]

\item\label{prop:newx} 
$x_j\notin B^{a(j,1)} \cup N_G(B^{a(j,1)})$ ~ ,

\item\label{prop:sizeBji} 
$|B^{x_j}_i|<n^{(1-i\eps)/3}$ ~ ,

\item\label{prop:edgesBji} 
$e(B^{x_j}_i)=0$ ~ ,

\item\label{prop:disjoint} 
$ B_i^{x_j} \cap \left( N_G(B^{a(j,1)}) \cup B^{a(j,1)} \right)  = \varnothing $ ~ ,

\item\label{prop:neighbours} 
if we define 
$N_{(j,i)}^s:= \left\{ v\in V\setminus B^{(j,i)} : d_G\left(v, B^{(j,i)} \right) \geq s\right\}$ then
$$
|N_{(j,i)}^s| \leq \left( 2j\eps^{-1} + i \right) n^{\frac{3-s(1+\eps)}{3}}~~ \text{for every }s\in \{0,1,2,3\}~ ,
$$ 
\end{enumerate}
and for every $k\in [t]$ we have
\begin{enumerate}[label=\itmarab{P},start=6]
\item\label{prop:rounds}
  $r_{x_k} = \tilde{r}_k$.
\end{enumerate}
\end{lemma}

Before proving Lemma~\ref{lem:tec}, let us first show how it implies Lemma~\ref{lem:bad}.

\begin{proof}[Proof of Lemma~\ref{lem:bad}]
Apply Lemma~\ref{lem:tec} with $t=7$. Then a.a.s.~we can find a set $A=\{x_1,\ldots,x_t\}$ as promised by this lemma.
Now, fix any set $M\subset V(G)$ of size 3. 
Since $|A|=7$, it will be enough to verify the following two statements.

\begin{enumerate}
\item[(i)] Every vertex $x\in A$ satisfies \ref{Bprop:neighbours1}--\ref{Bprop:nonbad}. 

\item[(ii)] At most six vertices $x\in A$ do not satisfy \ref{Bprop:disjointM}.
\end{enumerate}

For (i), consider any $x_j\in A$.
Property~\ref{Bprop:neighbours1} follows immediately by the definition of $B_1^{x_j}$
and Property~\ref{prop:edgesBji}.
Moreover, Property~\ref{Bprop:nonbad} follows immediately from the halt condition of Algorithm~\ref{alg:bad}.
To see Property~\ref{Bprop:neighbours2}, let $v\in B_i^{x_j}$. By the algorithm, $v$ is added to $B_i^{x_j}$ if
$d_G \left(v,\bigcup_{k< i} B_k^{x_j} \right)\geq 2 $.
Moreover, we have $v\in V\setminus B^{a(j,i)}$,
because of Property~\ref{prop:disjoint} and since 
${B_i^{x_j}\cap \left( \bigcup_{k<i} B_i^{x_j}\right)=\varnothing}$ according to the algorithm.
Now, using the Properties~\ref{prop:neighbours} and~\ref{prop:rounds},
and provided $n$ is large enough,
we deduce $|N_{a(j,i)}^3| < n^{-\eps/2}$
and thus $v\notin N_{a(j,i)}^3=\varnothing$. This yields
$d_G \left(v,\bigcup_{k<i} B_k^{x_j} \right)\leq 
d_G \left(v, B^{a(j,i)} \right) \leq 2 $.
Finally, using that $e_G(B_i^{x_j})=0$ 
according to Property~\ref{prop:edgesBji}, we deduce
$d_G \left(v,\bigcup_{k\leq i} B_k^{x_j} \right) = 2 $, proving~\ref{Bprop:neighbours2}.

Let us prove (ii) then. For any $k<j$, we have  $B^{x_k}\subset B^{a(j,1)}$ by Definition (\ref{eq:Bji}) and 
since $B^{x_k}=\bigcup_{i\leq r_{x_k}} B_{i}^{x_k}$ by Algorithm~\ref{alg:bad}.
Thus, using Property~\ref{prop:disjoint} we conclude that
$B^{x_j}$ and $B^{x_k}$ are disjoint. 
Moreover, since $B^{x_k}\subset B^{a(j,i)}$
we also obtain that $N_G(B^{x_k})\subset N_G(B^{a(j,i)})$.
Thus, using Property~\ref{prop:disjoint} again,
we get that $G$ does not have any edges between $B^{x_j}$ and $B^{x_k}$.
As a consequence we have that every vertex $v$ which is adjacent to
but not contained in $B^{x_j}$ for some $j\in [7]$
needs
to be element of $V\setminus B^{(t,r_{x_t})}$.
However, according to Property~\ref{prop:neighbours} and since
$r_{x_j}\leq \lceil 1/\eps \rceil$
holds by Property~\ref{prop:rounds},
we obtain $N_{(t,r_{x_t})}^3 = \varnothing$
for large enough $n$.
This implies that every vertex of
$V\setminus B^{(t,r_{x_t})}$
is adjacent to at most two of the sets
$B^{x_j}$ with $j\in [7]$.

We conclude that at most 3 of the pairwise disjoint sets $B^{x_j}$ may contain a vertex
of $M$. If a vertex $v\in M$ belongs to some set $B^{x_j}$ with $j\in [7]$, then $v\notin B^{x_k}\cup N_G(B^{x_k})$
for every $k\neq j$. If otherwise a vertex $v\in M$
belongs to $V\setminus B^{(t,r_{x_t})}$, then
it is adjacent to at most two of the sets $B^{x_j}$. 
Hence, there are at most six vertices $x\in A$
such that $M\cap (B^x \cup N_G(B^x))\neq \varnothing$. This proves statement~(ii).
\end{proof}

\begin{proof}[Proof of Lemma~\ref{lem:tec}]
For the proof of Lemma~\ref{lem:tec} we expose the edges
of $G\sim G_{n,p}$
step by step with respect to the given algorithm, and only during the process we choose the vertices of $A$ randomly.
To be more precise, we proceed as follows:
We first choose $x_1$ uniformly at random from $V(G)=[n]$ and then apply Algorithm~\ref{alg:bad} for $x_1$. Once, Algorithm~\ref{alg:bad} has been applied for $x_{j-1}$ and afterwards $B^{x_{j-1}}$ is determined, we choose $x_j$ uniformly at random from $[n]$ and
apply Algorithm~\ref{alg:bad} for $x_j$. While doing this, we always expose only those edges which have not been exposed yet and which are needed to determine the next set $B_{i}^{x_j}$ in the algorithm. For example:
When applying the algorithm for $x_1$,
we first expose only the edges incident to $x_1$ so that we are able to determine $B_1^{x_1}$. Once this set is fixed, we expose all edges incident to $B_1^{x_1}$ that have not been exposed yet, so that we can find $B_2^{x_1}$.
We then expose all edges incident to $B_2^{x_1}$
that have not been exposed yet, and so on.

For the analysis of the algorithm, we consider the pairs $(j,i)$, with $j\in [t]$ and $i\in [\tilde{r}_j]$, in lexicographic order. We consider the following event:\\

~~~~ \begin{tabular}{ll}
$E_{(j,i)}$: & ~~ for all pairs until and including $(j,i)$ the Properties~\ref{prop:newx}--\ref{prop:neighbours} hold,\\ 
 & ~~ and the Property~\ref{prop:rounds} is true for all $k<j$~ .
\end{tabular}

\vspace{0.5cm}

We will show that
\begin{equation}\label{goal}
\Prob\left( \overline{E_{(j,i)}} \big| E_{a(j,i)} \right) < 5n^{-\eps/3}
\end{equation} 
for every pair $E_{(j,i)}$,
where $E_{a(1,1)}$ is the event which is always true.
Before going into detail,
let us first prove that Lemma~\ref{lem:tec} follows, once
(\ref{goal}) is proven.

\begin{clm}\label{claim_final}
If (\ref{goal}) holds, then 
$\Prob\left( E_{(t,r_{x_t})} ~ \text{and} ~ r_{x_t}=\tilde{r}_t \right) \geq 1- n^{-\eps/4}$.
\end{clm}

\textit{Proof.}
Observe first that for every $j\in [t]$ the events 
$E_{(j,r_{x_j})}$ and $E_{(j,\tilde{r}_j)}$ are equivalent.
Indeed, by definition
$E_{(j,r_{x_j})}$ implies $E_{(j,\tilde{r}_j)}$, since $r_{x_j}\geq \tilde{r}_j$. 
Now, let $E_{(j,\tilde{r}_j)}$ be given
and let us explain why $E_{(j,r_{x_j})}$ follows then.
If we assume that the latter does not hold, then $\tilde{r}_j \neq r_{x_j}$,
and by definition of $\tilde{r}_j$ we then have
$\tilde{r}_j=\lceil 1/\eps \rceil < r_{x_j}$. Applying
\ref{prop:sizeBji} for $(j,\tilde{r}_j)$, which is given under assumption of 
$E_{(j,\tilde{r}_j)}$, we obtain $B_{\tilde{r}_j}^{x_j} = \varnothing$.
But this means that Algorithm~\ref{alg:bad}, when processed for vertex $x_j$, must have already stopped, i.e. $r_{x_j}<\tilde{r}_j$,
a contradiction. 

Moreover, by looking at the above argument more carefully we see that whenever one of the events $E_{(j,r_{x_j})}$ and $E_{(j,\tilde{r}_j)}$ holds,
we must have $r_{x_j}=\tilde{r}_j\leq \lceil 1/\eps \rceil$.

For every $j\in [t]$ we now conclude that
\begin{align*}
\Prob\left( \overline{E_{(j,r_{x_j})} } \right) 
&
= \Prob\left( \overline{E_{(j,\tilde{r}_j)}} \right)  
\leq \sum_{i=1}^{\tilde{r}_j} \Prob\left(
 \overline{E_{(j,i)}} \Big| E_{a(j,i)} \right)
 + \Prob\left( \overline{ E_{a(j,1)} } \right) \\
& 
\stackrel{(\ref{goal})}{\leq} \tilde{r}_j \cdot 5n^{-\frac{\eps}{3}} 
 + \Prob\left( \overline{ E_{(j-1,r_{x_{j-1}})} } \right) 
 <  \frac{10}{\eps} n^{-\frac{\eps}{3}} 
 + \Prob\left( \overline{ E_{(j-1,r_{x_{j-1}})} } \right) ~ .
\end{align*}
Applying the above inequality recursively we finally obtain
\begin{align*}
\Prob\left( \overline{E_{(t,r_{x_t})} ~ \text{and} ~ r_{x_t}=\tilde{r}_t } \right)  =
\Prob\left( \overline{E_{(t,r_{x_t})} } \right)  
< t\cdot \frac{10}{\eps} n^{-\frac{\eps}{3}} < n^{-\frac{\eps}{4}}  
\end{align*}
as claimed. \done

\medskip

It thus remains to prove (\ref{goal}). We start with a few observations.

\begin{obs}\label{obs:addvertex}
If Algorithm~\ref{alg:bad} adds a vertex $v$ to the set
$B_i^{x_j}$, then $d_G(v,B_{i-1}^{x_j})\geq 1$ and
\linebreak 
${d_G \left(v,\bigcup_{k\leq i-1} B_{k}^{x_j}\right)\geq 2}$.
\end{obs}

\textit{Proof.}
Algorithm~\ref{alg:bad} adds a vertex $v$ to $B_i^{x_j}$ 
if $d_G \left(v,\bigcup_{k\leq i-1} B_{k}^{x_j}\right)\geq 2$
and only if $v$ was not already added to some $B_{k}^{x_j}$
with $k<i$. However, the latter ensures
$d_G \left(v,\bigcup_{k\leq i-2} B_{k}^{x_j}\right)\leq 1$
and thus $d_G(v,B_{i-1}^{x_j})\geq 1$. \done

\medskip

\begin{obs}\label{obs:ifEaji}
If $E_{a(j,i)}$ holds, then the following is true:
\begin{enumerate}
\item[(i)] $\left| \bigcup_{k\leq i-1} B_k^{x_j} \right| 
\leq \left|B^{a(j,i)}\right| < n^{1/3}$
and $\left| \bigcup_{k\leq i-1} N_G(B_k^{x_j}) \right| 
\leq \left|N_G(B^{a(j,i)})\right| < n^{2/3-\eps}$.
\item[(ii)] If Algorithm~\ref{alg:bad} adds a vertex $v$ to $B_i^{x_j}$, then $v\in V\setminus B^{a(j,i)}$.
\end{enumerate}
\end{obs}

\textit{Proof.}
If $E_{a(j,i)}$ holds, we obtain
\begin{align*}
|B^{a(j,i)}| & 
\stackrel{ (\tref{eq:Bji}) }{\leq}
 \sum_{k<j} \sum_{\ell \leq r_k} |B_{\ell}^{x_k}|
+ \sum_{\ell<i} |B_{\ell}^{x_j}| 
 \stackrel{ \tref{prop:sizeBji} }{\leq} 
\sum_{k<j} \sum_{\ell \leq r_k} n^{\frac{1-\ell \eps}{3}}
+ \sum_{\ell<i} n^{\frac{1-\ell\eps}{3}} < 2jn^{\frac{1-\eps}{3}} < n^{\frac13}
\end{align*}
and
\begin{equation*}
|B^{a(j,i)}\cup N_G(B^{a(j,i)})| 
\stackrel{ (\ref{cl:deg}) }{\leq} 
2jn^{\frac{1-\eps}{3}} + 2jn^{\frac{1-\eps}{3}}\cdot 2n^{\frac13-\eps} < n^{\frac23-\eps}
\end{equation*}
provided $n$ is large enough. Thus, (i) follows.

For (ii) observe that,
according to the algorithm, no vertex from $\bigcup_{k\leq i-1} B_k^{x_j}$
can be added to $B_i^{x_j}$.
Moreover, using Property~\ref{prop:disjoint}, no vertex in $B^{a(j,1)}$ has a neighbour in $B_{i-1}^{x_j}$
(or in $\{x_j\}$ in the case when $i=1$, because of \ref{prop:newx}), while every vertex
being added to $B_i^{x_j}$ needs to have such a neighbour according to Observation~\ref{obs:addvertex}
(or since $B_1^{x_j}=N_G(x_j)$ if $i=1$).
It thus follows that no vertex from $B^{a(j,i)}=B^{a(j,1)} \cup \bigcup_{k\leq i-1} B_k^{x_j}$ is added to $B_i^{x_j}$.~\done

\medskip

Now, for each $j\in [t]$ and $i\in [\tilde{r}_j]$ we will prove (\ref{goal}),
by showing that, under condition of $E_{a(j,i)}$, each of the Properties~\ref{prop:newx}--\ref{prop:neighbours} in Lemma~\ref{lem:tec} fails to hold  for the pair $(j,i)$ with probability smaller than $n^{-\eps/3}$. This is obviously enough for showing 
(\ref{goal}) when $i>1$. To get (\ref{goal}) for $i=1$, 
recall that under condition of 
$E_{a(j,1)}=E_{(j-1,r_{j-1})}$ we also have that
$r_{x_{j-1}}=\tilde{r}_{j-1}$
(as shown in the proof of Claim~\ref{claim_final}), making sure that
\ref{prop:rounds} holds for $(j,1)$ as well.
We discuss each of the Properties~\ref{prop:newx}--\ref{prop:neighbours} seperately.

\medskip

{\bf Property \ref{prop:newx}: } 
The statement is trivially true for $j=1$. So, let $j>1$
and let us condition on $E_{a(j,1)}$.
The vertex $x_j$ is chosen uniformly at random from $[n]$
after Algorithm~\ref{alg:bad} has been applied for $x_1,\ldots,x_{j-1}$ and $B^{a(j,1)}$ was determined. 
Now, conditioned on $E_{a(j,1)}$, we have
$|B^{a(j,1)}\cup N_G(B^{a(j,1)})| < n^{2/3}$
due to  Observation~\ref{obs:ifEaji}.
It thus follows that
\begin{equation}\label{P1bound}
\Prob \left( \text{\ref{prop:newx} fails for $j$} ~
\big| ~ E_{a(j,1)} \right) < n^{-\frac13} <n^{-\frac{\eps}{3}} ~ .
\end{equation}

{\bf Property \ref{prop:sizeBji}: } Let $j\in [t]$.
Consider first the case when $i=1$. Then
$B_1^{x_j}=N_G(x_j)$ and using Claim~\ref{cl:deg} we have
\begin{equation*}
\Prob \left( \text{\ref{prop:sizeBji} fails for $B_{1}^{x_j}$} \big| E_{a(j,1)}\right) < \exp(-n^{\frac13-2\eps}) < n^{-\eps}~ .
\end{equation*}
So, let $i>1$ from now on and consider the moment 
immediately after $B_{i-1}^{x_j}$ was determined, i.e. when
all remaining edges incident to $B_{i-1}^{x_j}$
get exposed in order to determine $B_i^{x_j}$.

When we condition on $E_{a(j,i)}$, 
only vertices from $v\in V\setminus B^{a(j,i)}$ can be added to $B_i^{x_j}$ according to Observation~\ref{obs:ifEaji}.
Moreover, before $B_{i-1}^{x_j}$
was determined, for every vertex $v\in V\setminus B^{a(j,i)}$
all the edges towards $B_{i-1}^{x_j}$
have not been exposed so far.
Now, if a vertex $v\in V\setminus B^{a(j,i)}$ is added to $B_{i}^{x_j}$ then by Observation~\ref{obs:addvertex} one of the following two options needs to happen:
\begin{enumerate}
\item[(i)] $d_G(v,B_{i-1}^{x_j})\geq 2$, or
\item[(ii)] $d_G(v,B_{i-1}^{x_j})= 1$ and $d_G(v,\bigcup_{k<i-1}B_{k}^{x_j})\geq 1$.
\end{enumerate}

Conditioned on $E_{a(j,i)}$, the expected number of vertices in (i) is smaller than 
$n \cdot |B_{i-1}^{x_j}|^2\cdot  p^2
\stackrel{\tref{prop:sizeBji}}{<} n^{1/3-2(i+2)\eps/3}~ .
$
For (ii), observe that $d_G(v,\bigcup_{k<i-1}B_{k}^{x_j})\geq 1$ 
means
$v\in \bigcup_{k<i-1} N_G(B_{k}^{x_j})$.
Since we have
$
\left| \bigcup_{k<i-1} N_G\left(  B_{k}^{x_j} \right) \right|
< n^{2/3-\eps}
$
according to Observation~\ref{obs:ifEaji},
we get that 
the expected number of vertices in $V\setminus B^{a(j,i)}$ satisfying (ii) is at most
$
n^{2/3 - \eps}\cdot |B_{i-1}^{x_j}|\cdot p 
\stackrel{\tref{prop:sizeBji}}{<} n^{1/3-(i+5)\eps/3}~ .
$
Summing up, we get that the (conditional) 
expected size of $B_{i}^{x_j}$
is at most
$$
n^{\frac{1-2(i+2)\eps}{3}} + n^{\frac{1-(i+5)\eps}{3}}
< n^{\frac{1-(i+4)\eps}{3}}
$$
and thus, using Markov's inequality (Lemma~\ref{lem:markov}), we obtain
\begin{equation}\label{P2bound}
\Prob \left( \text{\ref{prop:sizeBji} fails for $(j,i)$} \big| E_{a(j,i)}\right) < n^{-\frac{4\eps}{3}} < n^{-\eps}~ .
\end{equation}

{\bf Property \ref{prop:edgesBji}: }
Again we condition on the event $E_{a(j,i)}$. 
By Observation~\ref{obs:addvertex}
all the vertices we add to $B_i^{x_j}$ need to come from
$V\setminus B^{a(j,i)}$. Thus, when $B_i^{x_j}$ 
is determined, neither of the edges in $E(B_i^{x_j})$ has been exposed before. With probability at least $1-n^{-4\eps/3}$
we get $|B_i^{x_j}|<n^{(1-i\eps)/3}$,
according to (\ref{P2bound}).
If we condition on the latter, the expectation of 
$e_G(B_{i}^{x_j})$ is smaller than
$
|B_{i}^{x_j}|^2 \cdot p  \stackrel{\tref{prop:sizeBji}}{<} n^{-4\eps/3} ~ .
$
Thus, using Markov's inequality (Lemma~\ref{lem:markov}),
\begin{equation*}
\Prob\left(\text{\ref{prop:edgesBji} fails for $(j,i)$} \big| E_{a(j,i)} \right) \leq n^{-\frac{4\eps}{3}} + n^{-\frac{4\eps}{3}} < n^{-\frac{\eps}{3}}~ .
\end{equation*}

{\bf Property \ref{prop:disjoint}:}
Let $i=1$. If $j=1$ then the statement is trivially true. Otherwise, we know from (\ref{P1bound})
that, under condition of $E_{a(j,i)}$, we have
$x_j\notin B^{a(j,1)} \cup N_G(B^{a(j,1)})$ with probability at least $1-n^{-1/3}$.
This implies $B_1^{x_j}\cap B^{a(j,1)}=\varnothing$
and thus it remains to check that it is unlikely to have a vertex from 
$N_G(B^{a(j,1)})\setminus B^{a(j,1)}$ landing in $B_1^{x_j}$.
Note that before $B_1^{x_j}$ gets determined none
of the edges between $N_G(B^{a(j,1)})\setminus B^{a(j,1)}$ and $x_j$ has been exposed so far, when
$x_j\notin B^{a(j,1)} \cup N_G(B^{a(j,1)})$. Thus, using
Observation~\ref{obs:ifEaji},
we (conditionally) expect at most
$
|N_G(B^{a(j,1)})|\cdot p 
< n^{2/3-\eps} \cdot p = n^{-2\eps}
$
vertices in $\left( N_G(B^{a(j,1)})\setminus B^{a(j,1)}\right) 
\cap B_1^{x_j}$. It follows that  
$$
\Prob\left(\text{\ref{prop:disjoint} fails for $(j,1)$} \big| E_{a(j,1)} \right)
 < n^{-\frac13} + n^{-2\eps} < n^{-\eps}~ .
$$

Let $i>1$ then. Under assumption of $E_{a(j,i)}$, we have that
$B_{i-1}^{x_j} \cap N_G(B^{a(j,1)}) = \varnothing$
according to \ref{prop:disjoint}.
But then, according to Observation~\ref{obs:addvertex}, no vertex from
$B^{a(j,1)}$ is added to $B_i^{x_j}$, giving that
$B_i^{x_j}\cap B^{a(j,1)} = \varnothing$.
It thus remains to check that it is unlikely to have a vertex from 
$N_G(B^{a(j,1)})\setminus B^{a(j,1)}$ landing in $B_i^{x_j}$.
Using that $B_{i-1}^{x_j}\cap B^{a(j,1)} = \varnothing$ by \ref{prop:disjoint}, we note that before $B_i^{x_j}$ gets determined, no edge between $N_G(B^{a(j,1)})\setminus B^{a(j,1)}$ and $B_{i-1}^{x_j}$ has been exposed.
Now, applying Observation~\ref{obs:addvertex},
a vertex $v$ from $N_G(B^{a(j,1)})\setminus B^{a(j,1)}$ 
is added to $B_{i}^{x_j}$ 
if 
\begin{enumerate}
\item[(i)] $d_G(v,B_{i-1}^{x_j})\geq 2$, or
\item[(ii)] $d_G(v,B_{i-1}^{x_j})= 1$ and $d_G(v,\bigcup_{k<i-1}B_{k}^{x_j})\geq 1$.
\end{enumerate}

Hereby, again using Observation~\ref{obs:ifEaji} as well as \ref{prop:sizeBji}, the (conditional) expected number of vertices in (i) is at most
$
|N_G(B^{a(j,i)})|\cdot |B_{i-1}^{x_j}|^2 p^2 < n^{-3\eps}~ .
$
For (ii), observe that 
$N_G(B^{a(j,1)})\setminus B^{a(j,1)}\subset V \setminus B^{a(j,i)}$ holds, since
$\bigcup_{k<i}B_{k}^{x_j}$ and $N_G(B^{a(j,1)})$ are disjoint due to Property~\ref{prop:disjoint}.
Therefore, if
$d_G(v,\bigcup_{k<i-1}B_{k}^{x_j})\geq 1$ 
and $v\in N_G(B^{a(j,1)})\setminus B^{a(j,1)}$, then
$v\in N_{a(j,i)}^2$.
Using \ref{prop:neighbours} and~\ref{prop:rounds}, we have that 
$\left| N_{a(j,i)}^2 \right|< n^{1/3}$.
Thus, the (conditional) expected number of vertices 
satisfying (ii) is bounded from above by
$
n^{1/3}\cdot |B_{i-1}^{x_j}|\cdot p 
\stackrel{\tref{prop:sizeBji}}{<} n^{-4\eps/3}~ .
$
Summing up, we expect at most
$$
n^{-3\eps} + n^{-\frac{4\eps}{3}} < n^{-\eps} 
$$ 
vertices in $B_i^{x_j}\cap (N_G(B^{a(j,1)}) \setminus B^{a(j,1)} )$.
By Markov's inequality (Lemma~\ref{lem:markov}) we obtain
\begin{equation*}
\Prob\left(\text{\ref{prop:disjoint} fails for $(j,i)$} \big| E_{a(j,i)} \right)
 < n^{-\eps}~ .
\end{equation*}

{\bf Property~\ref{prop:neighbours}:}
Consider first the case when $(j,i)=(1,1)$. 
The bound on $|N_{(1,1)}^0|$ is trivially true.
So, let $s\geq 1$. Immediately after $B^{(1,1)}=N_G(x_1)$ is determined, 
none of the edges between $V\setminus B^{(1,1)}$
and $B^{(1,1)}$ has been exposed. Moreover, according to Lemma~\ref{lem:degree_bound}, with 
probability at least $1-\exp(-n^{1/3-2\eps})$
we have $|B^{(1,1)}|<n^{(1-\eps)/3}$.
Thus, if we condition on that bound, the expected size
of $N_{(1,1)}^s$ is at most
$
n\cdot \left( |B^{(1,1)}|\cdot p \right)^s
< n^{(3-s(1+4\eps))/3}~ .
$
It follows that 
\begin{align*}
\Prob\Big(\text{\ref{prop:neighbours} fails for $(1,1)$} \Big) & 
\leq \sum_{s\in [3]} \frac{n^{\frac{3-s(1+4\eps)}{3}}}{\left( 2\eps^{-1} + 1 \right) n^{\frac{3-s(1+\eps)}{3}}} 
< n^{-\frac{\eps}{3}}~ .
\end{align*}

So let $(j,i)\neq (1,1)$ from now on. 
Again, the bound on $|N_{(j,i)}^0|$ is trivially true.
Under condition of
$E_{a(j,i)}$ we have
$
|B_i^{x_j}| < n^{(1-i\eps)/3}
$
with probability  at least $1-n^{-4\eps/3}$,
according to (\ref{P2bound}). Condition on the latter from now on. Given that $E_{a(j,i)}$ holds, we additionally get
\begin{align*}
\left| N_{a(j,i)}^s \right|
\stackrel{\tref{prop:neighbours}}{\leq}
\begin{cases}
\left( 2j\eps^{-1} + i-1 \right) n^{\frac{3-s(1+\eps)}{3}}
	& ~ \text{if }i> 1\\
\left( 2(j-1)\eps^{-1} + r_{x_{j-1}} \right) n^{\frac{3-s(1+\eps)}{3}} 
\stackrel{\tref{prop:rounds}}{<} \left( 2j\eps^{-1} + i-1 \right) n^{\frac{3-s(1+\eps)}{3}}
	& ~ \text{if }i = 1
\end{cases}
\end{align*}
for every $s\in \{0,1,2,3\}$.  Thus,
\begin{align}\label{neighbour_step1}
\left| N_{(j,i)}^s \right|
& = \left| N_{(j,i)}^s \cap N_{a(j,i)}^s \right| 
+ \left| N_{(j,i)}^s \setminus N_{a(j,i)}^s \right| \nonumber \\
& \leq  \left( 2j\eps^{-1} + i-1 \right) n^{\frac{3-s(1+\eps)}{3}}
+ \left| N_{(j,i)}^s \setminus N_{a(j,i)}^s \right|~ .
\end{align}
Now, for $s\in [3]$, if a vertex $v$ ends up being in
$N_{(j,i)}^s \setminus N_{a(j,i)}^s$
then, by definition, $v\in V\setminus B^{(j,i)} \subset V\setminus B^{a(j,i)}$ and $d_G\left( v, B^{a(j,i)} \right) = t$
for some $t<s$. But this means that $v\in N_{a(j,i)}^t$
and, in order to be added to $N_{(j,i)}^s$,
the vertex $v$ needs to get at least $s-t$ edges
towards $B_i^{x_j}$ (which get exposed only after $B_i^{x_j}$ has been determined, since $v\in V\setminus B^{(j,i)}$). We conclude that the (conditional) expected size of 
$\left| N_{(j,i)}^s \setminus N_{a(j,i)}^s \right|$ is at most
\begin{align*}
\sum_{t<s} \left| N_{a(j,i)}^t \right| \cdot 
	\left( \left| B_i^{x_j} \right|\cdot p \right)^{s-t}
	& \leq
\left( 2j\eps^{-1} + i-1 \right) 
\sum_{t<s} n^{\frac{3-t(1+\eps)}{3}}
\left( n^{\frac{-1-(i+3)\eps}{3}} \right)^{s-t} \\
	& =
\left( 2j\eps^{-1} + i-1 \right) 
\sum_{t<s} n^{\frac{3-s(1+\eps) - (i+2)(s-t)\eps}{3}}\\
	& \leq
\left( 2j\eps^{-1} + i-1 \right) 
\sum_{t<s} n^{\frac{3-s(1+\eps) - 3\eps}{3}}
	 < n^{\frac{3-s(1+\eps) - 2.5\eps }{3}},
\end{align*}
where the first inequality uses \ref{prop:sizeBji} and \ref{prop:neighbours}, and the last inequality uses
that $i\leq \tilde{r}_j\leq 2\varepsilon^{-1}$.
Thus, with Markov's inequality (Lemma~\ref{lem:markov}) and union bound, we obtain
\begin{align*}
\Prob\left( \exists s\in[3]:~ 
\left| N_{(j,i)}^s \setminus N_{a(j,i)}^s \right|
> n^{\frac{3-s(1+\eps)-\eps}{3}} \right) < n^{-\frac{\eps}{2}} ~ .
\end{align*}
Combining this with (\ref{neighbour_step1}),
we see that with (conditional) probabilitiy at least
$1-n^{-\eps/2}$ we have
$$
\left| N_{(j,i)}^s \right|
\leq  \left( 2j\eps^{-1} + i-1 \right) n^{\frac{3-s(1+\eps)}{3}}
+ n^{\frac{3-s(1+\eps)-\eps}{3}} 
< \left( 2j\eps^{-1} + i \right) n^{\frac{3-s(1+\eps)}{3}}
$$
for every $s\in [3]$, and thus
\begin{equation*}
\Prob\left(\text{\ref{prop:neighbours} fails for $(j,i)$} \big| E_{a(j,i)} \right)
 < n^{-\frac{4\eps}{3}} + n^{-\frac{\eps}{2}} < n^{-\frac{\eps}{3}}~ . 
\end{equation*}

This finishes the proof of Lemma~\ref{lem:tec}.
\end{proof}

\section{Good structures for Connector}
\label{sec:structures}

\subsection{Technical Lemma}

Throughout Section~\ref{sec:structures}, we will consider the sequence $(\alpha_i)_{i\in\mathbb{N}}$ given by 
\begin{equation}\label{def:alpha}
\alpha_i:=3\cdot 2^{i-1} - 2~ ,
\end{equation}
and note that $\alpha_1=1$ and
$\alpha_{i+1}=2(\alpha_i+1)$ hold.
In order to simplify our argument, in the next lemma we will consider $\eps$ to be a real number of the form
$\eps = 1/(9 \cdot 2^{k-2}-3)$
with $k\in\mathbb{N}$. Note that this yields
\begin{equation}\label{alpha}
\alpha_k \eps = \frac{2}{3}~~ \text{and} ~~
\alpha_{k-1} \eps = \frac{1}{3} - \eps~ .
\end{equation}
Given $G\sim G_{n,p}$ and $x\in V(G)$, with high probability
the following lemma  will provide us with a suitable subgraph $H$
of $G$ which later (see e.g. Claim~\ref{cl:Tv}) 
will turn out to contain a bunch of 
copies of $\Tk$ that help to prove
Lemma~\ref{lem:tec1_connector} and Lemma~\ref{lem:tec2_connector}.

In order to simplify notation, we set 
$$\Ik=\left\{ (i,j,\ell):~ 1 \leq i \leq k, 1 \leq j \leq 2^{k-i}, 1 \leq \ell \leq 4 \right\}.$$
Starting with vertex disjoint sets $V_{(i,j,\ell)}\subset V(G)$
for $(i,j,\ell)\in \Ik$
we will iteratively find well-behaving subsets 
$M_{(i,j,\ell)}$ of those
which later turn out to be good candidate sets for embedding
the vertices of $\mathcal{T}_k$, even when some edges
of $G$ are not allowed to be used. Hereby, the tuple $(i,j)$ will represent the position of a vertex in the desired tree, while the component $\ell$ is used in order to apply our argument on disjoint subsets of $V(G)$ labeled with distinct indices $\ell$,
so that we will be able to find a few edge-disjoint copies of $\Tk$.  

\begin{lemma}[Decomposition of $G_{n,p}$]\label{lem:Decomp}
Let $k\geq 3$ be an integer, let
$\eps =(9 \cdot 2^{k-2}-3)^{-1}$
and let $n\in\mathbb{N}$ be large enough.
Let $G\sim G_{n,p}$ with $p=n^{- 2/3 +\eps}$.
Fix $x \in V(G)$
and~let
$$
V(G) = \{x\} \cup \bigcup\limits_{(i,j,\ell) \in \Ik} V_{(i,j,\ell)} \cup R
$$
be any partition of $V(G)$ such that 
$|V_{(i,j,\ell)}|=n/(2^{k+4})$ holds. 
Then with probability at least $1-\exp(-\ln^{1.5}n)$ there exist sets $M_{(i,j,\ell)}\subset V(G)$ and a subgraph $H \subset G$ with 
$$V(H)=\bigcup\limits_{\substack{(i,j,\ell) \in \Ik}} M_{(i,j,\ell)} \cup \{x\}$$ 
such that the following is true for every $(i,j,\ell) \in \Ik$:
\begin{enumerate}[label=\itmarab{D}]
  \item\label{decomp:inclusion} $M_{(i,j,\ell)} \subset V_{(i,j,\ell)}$
  \item\label{decomp:size} 
  $|M_{(i,j,\ell)}|=n^{1/3+\alpha_i \eps} \ln^{- 2\alpha_i} n$
  \item\label{decomp:cherries} if $i\geq 2$ then $\forall v \in M_{(i,j,\ell)}: d_H\left(v,M_{(i-1,2j-1,\ell)}\right) \geq 1$ and $d_H\left(v,M_{(i-1,2j,\ell)}\right) \geq 1$
  \item\label{decomp:degreebound} if $i\geq 2$ then $\forall v \in M_{(i-1,2j-1,\ell)}\cup M_{(i-1,2j,\ell)}: d_H(v,M_{(i,j,\ell)}) \leq n^{(\alpha_{i}- \alpha_{i-1})\eps}$
	\item\label{decomp:firstneighbours} $\forall v \in M_{(1,j,\ell)}: x \in N_H(v)$
	\item\label{decomp:edges} $\forall e \in E(H) ~ \exists (i,j,\ell)\in \Ik: e\in E_G(M_{(i-1,2j-1,\ell)} \cup M_{(i-1,2j,\ell)},M_{(i,j,\ell)}) $
\end{enumerate}
where we define $M_{(0,j,\ell)}:=\{x\}$ for
every $j\in [2^k]$ and $\ell\in [4]$.
\end{lemma}

\begin{figure}[h]
\begin{center}
\includegraphics[scale=1]{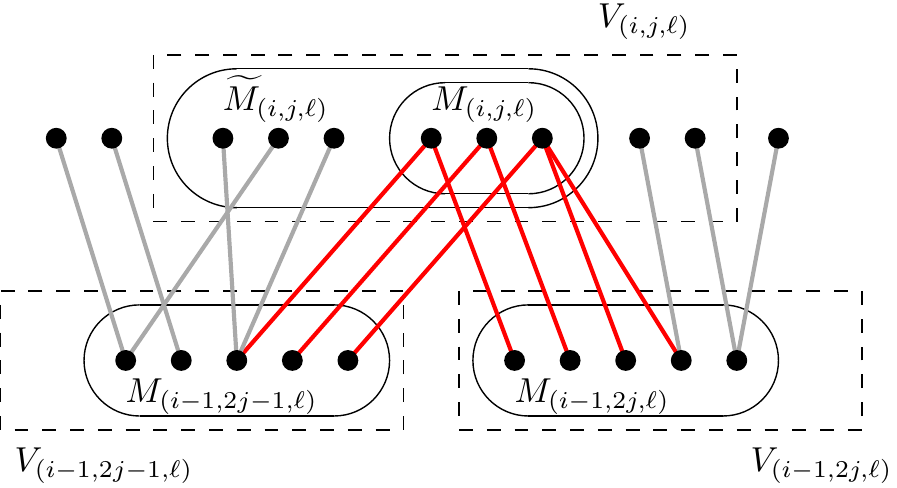}
\caption{Part of the structure of $H$, depicted in {\color{red} red}}
\label{fig:Mijl}
\end{center}
\end{figure}

\begin{proof}
Fix $x \in V(G)$
and~let
$$
V(G) = \{x\} \cup \bigcup\limits_{(i,j,\ell) \in \Ik} V_{(i,j,\ell)} \cup R
$$
be any partition of $V(G)$ such that 
$|V_{(i,j,\ell)}|=n/(2^{k+4})$. 
We will iteratively construct a subgraph $H$ together with sets $M_{(i,j,\ell)}\subset V_{(i,j,\ell)}$ through Algorithm~\ref{alg:good}
and we will prove that with probability at least $1-\exp(-\ln^{1.5}n)$
the algorithm succeeds in creating these in such a way that all the Properties~\ref{decomp:inclusion}--\ref{decomp:edges}
hold. Again, we will expose the edges of $G\sim G_{n,p}$ while the algorithm is running. That is, whenever a new set is going to be determined, we will only expose those edges 
which have not been exposed before and which are needed for the corresponding step in the algorithm.

\begin{algorithm}[ht]\label{alg:good}
    \caption{Good subgraph $H$ for vertex $x$}    
\SetKwInOut{Input}{Input}
\SetKwInOut{Output}{Output}
\Input{~ graph $G$, vertex $x\in V(G)$, partition of $V(G)$ as described in Lemma~\ref{lem:Decomp}}
\Output{~ subgraph $H$, sets $M_{(i,j,\ell)}$}
	~ $M_{(0,j,\ell)} := \{x\}$ for every $j\in [2^k]$ and $\ell\in [4]$\;
  ~ $V(H) := \{x\}$ and $E(H):=\varnothing$ \; 
    \For{$1 \leq i \leq k$}{
		\For{$1 \leq j \leq 2^{k-i}$}{
		\For{$1 \leq \ell \leq 4$}{
\quad\mbox{} $M_{(i,j,\ell)}:=\{v \in V_{(i,j,\ell)}:~  d_G(v,M_{(i-1,2j-1,\ell)}) \geq 1, ~
d_G(v,M_{(i-1,2j,\ell)}) \geq 1 \}$ \;
		\lIf{$|M_{(i,j,\ell)}| \geq 
		n^{1/3+\alpha_i \eps} \ln^{-2\alpha_i} n$\\ 
         \quad\mbox{}}{
        remove randomly selected vertices from $M_{(i,j,\ell)}$\\ \hspace{16.5mm} until $|M_{(i,j,\ell)}| = 
        n^{1/3+\alpha_i \eps}\ln^{-2\alpha_i} n$}
	\quad\mbox{} $V(H) \leftarrow V(H) \cup M_{(i,j,\ell)}$ \\
	\quad\mbox{} $E(H) \leftarrow E(H) \cup E_G(M_{(i-1,2j-1,\ell)}\cup M_{(i-1,2j,\ell)},M_{(i,j,\ell)})$ \\
				}}}
				{halt with output $H$ and sets $M_{(i,j,\ell)}$}
  \end{algorithm}

Following Algorithm~\ref{alg:good} it is obvious that
the Properties~\ref{decomp:inclusion}, 
\ref{decomp:cherries}, \ref{decomp:firstneighbours} and \ref{decomp:edges} hold.

Let $\E_t$ be the event that the Properties~\ref{decomp:size} and \ref{decomp:degreebound} hold for every $i\leq t$
(and every $j\leq 2^{k-i}$ and $\ell\in [4]$).
In the following we will show that
$$
\Prob(\overline{\E_1})
	\leq \exp\left( - n^{\frac{1}{3}} \right)
~~~ \text{and} ~~~
\Prob\left(\overline{\E_t} \Big| \E_{t-1} \right)
	\leq \exp\left( - 2\ln^{1.5} n \right)
$$
for every $2\leq t\leq k$. Observe that, once these two inequalities are proven, we can deduce that
$\Prob(\E_k)\geq {1 - k \exp\left( - 2\ln^{1.5} n \right)} 
	\geq 1- \exp\left( - \ln^{1.5} n \right)$,
from which Lemma~\ref{lem:Decomp} follows.

\medskip 

\underline{The event $\E_1$:} Property~\ref{decomp:degreebound}
holds trivially when $i=1$. For \ref{decomp:size}
observe that a standard Chernoff argument (apply Lemma~\ref{lem:Chernoff1} and union bound) yields that,
with probability at least $1-\exp(- n^{1/3})$,
for every $j\in [2^{k-1}]$ we have
$$|N_G(x,V_{(1,j,\ell)})| 
\geq \frac{1}{2} p|V_{(1,j,\ell)}|
\geq \frac{n^{\frac13+\eps}}{2^{k+5}} \geq 
 n^{\frac13+\eps} \ln^{-2} n~ .$$
Since $M_{(1,j,\ell)}$ is obtained from 
$N_G(x,V_{(1,j,\ell)})$ by reducing the latter to size 
$n^{1/3+\eps} \ln^{-2} n$, we  then obtain
that \ref{decomp:size} holds for $i=1$. Thus, $\Prob(\overline{\E_1})\leq 
\exp(- n^{1/3})$.

\medskip

\underline{The event $\E_t$:} Assume that $\E_{t-1}$ holds.
Observe that, before the sets $M_{(i,j,\ell)}$ get determined by Algorithm~\ref{alg:good},
none of the edges between the sets $V_{(i,j,\ell)}$, with $j\leq 2^{k-i}$,
and the sets
$M_{(i-1,j',\ell)}$, with $j'\leq 2^{k-i+1}$, 
have been revealed so far.
Now, conditioning on Property~\ref{decomp:size} for $i=t-1$, 
a Chernoff-type argument yields the following
with probability at least 
$1-\exp\left( - \ln^{1.8} n \right)$:
\begin{enumerate}
\item[(i)] 
$d_G(v,M_{(i-1,2j-1,\ell)})\leq \ln^{1.9} n$ and
$d_G(v,M_{(i-1,2j,\ell)})\leq \ln^{1.9} n$
for every $v\in V_{(i,j,\ell)}$ with $j\leq 2^{k-i+1}$,
\item[(ii)]
$e_G(M_{(i-1,2j-1,\ell)}, V_{(i,j,\ell)}) 
\geq 0.5 \cdot |M_{(i-1,2j-1,\ell)}|\cdot |V_{(i,j,\ell)}| \cdot p
$
for every $j\leq 2^{k-i+1}$~ .
\end{enumerate}
In fact, for (i) observe that 
$|M_{(i-1,2j-1,\ell)}|=|M_{(i-1,2j,\ell)}|
\stackrel{\tref{decomp:size}}{=}
n^{1/3+\alpha_{i-1}\eps}\stackrel{(\ref{alpha})}{\leq} p^{-1}$ which implies that the (conditional) expectation of the degrees in (i) is bounded by $1$. Thus, applying Lemma~\ref{lem:Chernoff2} ensures that the probability for one vertex $v$ failing to satisfy the degree conditions in (i) is bounded
by $2j\exp(-\ln^{1.9}n)$; a union bound completes the argument. Moreover, by Lemma~\ref{lem:Chernoff1} and a simple union bound, the property in (ii) fails with probability at most
$\exp\left(-n^{2/3}\right)$.

Now, let us condition on the properties in (i) and (ii) from now on.
Consider 
\begin{align*}
\widetilde{M}_{(i,j,\ell)}
& := \left\{v\in V_{(i,j,\ell)}:~
d_G(v,M_{(i-1,2j-1,\ell)})\geq 1 \right\} ~ ,\\
\widehat{M}_{(i,j,\ell)}
& := \left\{v\in \widetilde{M}_{(i,j,\ell)}:~
d_G(v,M_{(i-1,2j,\ell)})\geq 1 \right\} 
\end{align*}
and observe that, according to Algorithm~\ref{alg:good}, $M_{(i,j,\ell)}$
has size $n^{1/3+\alpha_i\eps}\ln^{-2\alpha_i} n$
if and only if 
$\widehat{M}_{(i,j,\ell)}$
is at least of that size.
In order to see that the latter is likely to hold, 
we first observe that
\begin{align*}
|\widetilde{M}_{(i,j,\ell)}|
\stackrel{\text{(i)}}{\geq}
\frac{e_G(M_{(i-1,2j-1,\ell)}, V_{(i,j,\ell)})}{\ln^{1.9} n}
\stackrel{\text{(ii)},\tref{decomp:size}}{\geq}
n^{\frac{2}{3}+(\alpha_{i-1}+1)\eps}\ln^{-2\alpha_{i-1}-2} n~ .
\end{align*}
Notice that when we determine the set 
$\widetilde{M}_{(i,j,\ell)}$ we don't need 
to reveal the edges
between $\widetilde{M}_{(i,j,\ell)}$ and 
$M_{(i-1,2j,\ell)}$. Thus, we can expose these edges afterwards, and by a Chernoff-type argument
(apply Lemma~\ref{lem:Chernoff1} and union bound)
we get, with probability at least $1-\exp\left(-n^{1/3}\right)$, that
$$e_G(\widetilde{M}_{(i,j,\ell)} , M_{(i-1,2j,\ell)}) 
\geq 0.5 \cdot |\widetilde{M}_{(i,j,\ell)}|
	\cdot |M_{(i-1,2j,\ell)}| \cdot p
$$
for every $j\leq 2^{k-i+1}$.  It then follows that
\begin{align*}
|\widehat{M}_{(i,j,\ell)}|
\stackrel{\text{(i)}}{\geq}
\frac{e_G(\widetilde{M}_{(i,j,\ell)} , M_{(i-1,2j,\ell)})}{\ln^{1.9} n}
\stackrel{\tref{decomp:size}}{\geq}
n^{\frac{1}{3}+(2\alpha_{i-1}+2)\eps}\ln^{-4\alpha_{i-1}-4} n
\stackrel{(\tref{def:alpha})}{=}
n^{\frac{1}{3}+\alpha_i\eps}\ln^{-2\alpha_i} n
\end{align*}
from which \ref{decomp:size} follows,
as was explained above.

As next, let us look at Property~\ref{decomp:degreebound}.
Fix $s\in \{0,1\}$ and consider the set 
$\widetilde{M}_{v,(i,j,\ell)}:=N_G(v,V_{(i,j,\ell)})$
for every $v\in M_{(i-1,2j-s,\ell)}$.
According to a Chernoff-type argument
(apply Lemma~\ref{lem:Chernoff1} and union bound), 
with probability at least
$1-\exp\left(- n^{1/3}\right)$, 
it holds that
$
|\widetilde{M}_{v,(i,j,\ell)})| 
	\leq n^{1/3 + \eps}
$ 
for every $v\in M_{(i-1,2j-s,\ell)}$.
Notice as before that, when we determine the set 
$\widetilde{M}_{v,(i,j,\ell)}$ we don't need 
to reveal the edges
between $\widetilde{M}_{v,(i,j,\ell)}$ and 
$M_{(i-1,2j-(1-s),\ell)}$. Thus, we can expose these edges afterwards, and another Chernoff-type argument
(apply Lemma~\ref{lem:Chernoff1} and union bound)
yields that,
with probability at least 
$1-n\exp\left( - n^{(\alpha_{i-1}+2)\eps}\right) 
\geq 1 - n^{-\eps}$,
it holds that 
$$e_G\left(\widetilde{M}_{v,(i,j,\ell)},M_{(i-1,2j - (1-s),\ell)}\right) 
\leq 0.5p|\widetilde{M}_{v,(i,j,\ell)}|\cdot |M_{(i-1,2j - (1-s),\ell)}|
\stackrel{\tref{decomp:size}}{\leq} n^{(\alpha_{i-1}+2)\eps} 
	\stackrel{(\tref{def:alpha})}{=} 
	n^{(\alpha_i - \alpha_{i-1})\eps }$$
for every $v\in M_{(i-1,2j-s,\ell)}$. 
Finally, notice that a vertex $w\in M_{(i,j,\ell)}\subset V_{(i,j,\ell)}$ can only become a neighbour of $v$ in $H$ 
if $w\in \widetilde{M}_{v,(i,j,\ell)}$ and 
$E_G(w,M_{(i-1,2j-(1-s),\ell)})\neq \emptyset$, 
where the first condition
comes from 
the definition of $\widetilde{M}_{v,(i,j,\ell)}$,
and the second is a consequence of
the construction of $M_{(i,j,\ell)}$
in the Algorithm~\ref{alg:good}.
But this immediately implies
$$
d_H(v,M_{(i,j,\ell)}) 
	\leq  e_G\left(\widetilde{M}_{v,(i,j,\ell)} , 
		M_{(i-1,2j - (1-s),\ell)}\right) 
	\leq n^{(\alpha_i - \alpha_{i-1})\eps } ~ ,
$$ 
as is required for Property~\ref{decomp:degreebound}.

Hence, summing up all the failure probabilities that occured in our argument, we see that 
\begin{align*}
\Prob\left(\overline{\E_t} \Big| \E_{t-1} \right)
& \leq \exp\left(-\ln^{1.8}n \right) 
	+ \exp\left(-n^{\frac{1}{3}} \right)
	+ 2\left( 
		\exp\left(-n^{\frac{1}{3}} \right)
		+ \exp\left(-n^{-\eps} \right) \right) \\
& < \exp\left( - 2\ln^{1.5} n \right)~ .
\end{align*}
This finishes the proof of Lemma~\ref{lem:Decomp}.
\end{proof}

Having Lemma~\ref{lem:Decomp} in our hands, we are now able
to prove Lemma~\ref{lem:tec1_connector} and Lemma~\ref{lem:tec2_connector} in the following.
For short, let us write
\begin{equation}\label{def:Ml}
M_{\ell} :=\bigcup\limits_{i,j:~ (i,j,\ell) \in \Ik} M_{(i,j,\ell)}
\end{equation}
for every $\ell\in [4]$,
and for every $i\in [k]$ set  
\begin{equation}\label{def:Li}
L_i :=\bigcup\limits_{j,\ell:~ (i,j,\ell) \in \Ik} M_{(i,j,\ell)}
	~~ \text{and} ~~ L_0:=\{x\}~ .
\end{equation}

\medskip

\subsection{Finding good structures -- Part I}

\begin{proof}[Proof of Lemma~\ref{lem:tec1_connector}]
Let $\delta>0$ be given. Then fix $k\in\mathbb{N}$ such that
$$
\delta\geq \eps := \frac{1}{9\cdot 2^{k-2}-3}
$$
holds, and let $k_1=k+1$. 
We will prove the lemma for
$
p=n^{-2/3+\eps}
$
and notice that by the monotonicity
the lemma then follows for $p=n^{-2/3+\delta}$ as well.
As before, we set
$$\Ik=\left\{(i,j,\ell)\in \mathbb{N}^3:~ 
1\leq i\leq k,~ 1\leq j\leq 2^{k-i},~ 1\leq \ell\leq 4
\right\}~ .$$
Let $V=[n]$ be the vertex set, and let $x,r\in V$ be fixed.
Before exposing all the edges of $G\sim G_{n,p}$
we fix a partition 
$$
[n]=\{x\}\cup \bigcup_{(i,j,\ell)\in \Ik} V_{(i,j,\ell)} \cup R~ 
$$
of the vertex set such that
$r\in R$ and $|V_{(i,j,\ell)}|=n/(2^{k+4})$
for every $(i,j,\ell)\in \Ik$. 
We  will show that with probability
at least $1-n^{-3}$ 
a random graph $G\sim G_{n,p}$ 
is such that for every subgraph 
$B$ with $e(B)\leq n^{1/3}\ln n$ 
the graph $G\setminus B$  contains a copy $T$ 
of $\mathcal{T}_{k_1}$ as desired, 
additionally satisfying that 
$V(T)\subset ([n]\setminus R)\cup \{r\}$.
Taking a union bound over all choices of $r$ and $x$, 
it then follows that the property described in 
Lemma~\ref{lem:tec1_connector} holds with probability
at least
$
1-n^{-1}~ .
$ 

In order to do so, we first expose the edges of $G$ on $[n]\setminus R$.
By Lemma~\ref{lem:Decomp} we know that with probability at least $1-\exp(-\ln^{1.5} n)$ there exist subsets 
$M_{(i,j,\ell)}\subset V_{(i,j,\ell)} $ 
and a subgraph $H\subset G$ on the vertex set 
$\bigcup_{(i,j,\ell)\in \Ik} M_{(i,j,\ell)}\cup \{x\}$
such that all the Properties~\ref{decomp:inclusion}--\ref{decomp:edges} hold.
Let $\E_H$ be the event that such a graph $H$ with vertex sets $M_{(i,j,\ell)}$ exists.
From now on, we will condition on $\E_H$ to hold. 
Recall the definition of $M_{\ell}$ in (\ref{def:Ml}) and $L_i$ in (\ref{def:Li}).
We first observe that there must be many copies of 
$\Tk$ in $H$ with all leaves being adjacent to $x$ in $H$.

\begin{clm}\label{cl:Tv}
For every $\ell\in [4]$ and every
$v\in M_{(k,1,\ell)}$ there exists a tree $T_v \cong \Tk$
such that 
\begin{enumerate}[label=\itmarab{T}]
\item\label{tree:root} $v$ is the root of $T_v$, 
\item\label{tree:VE} $V(T_v)\subset M_{\ell}$ and $E(T_v)\subset E(H)$,
\item\label{tree:leaves} all the leaves of $T_v$ are adjacent to $x$ in $H$,
\item\label{tree:children} for all vertices in $V(T_v)\cap L_i$ and $i\geq 2$, their children in $T_v$ belong to $L_{i-1}$.
\end{enumerate}
\end{clm}  

\textit{Proof.}
Label the vertices of $\Tk$ in such a way that
the root gets label $(k,1)$ and for every vertex
with label $(i,j)$ and $i>1$ its two children
get the labels 
$(i-1,2j-1)$ and $(i-1,2j)$, respectively.
 
Let $\ell\in [4]$ and $v\in M_{(k,1,\ell)}$ be given.
Applying Property~\ref{decomp:cherries} iteratively 
we find an embedding of $\Tk$ into $H$,
such that the root of $\Tk$ is mapped to $v$,
and such that the vertex with label $(i,j)$ in $\Tk$
is mapped to a vertex in $M_{(i,j,\ell)}\subset M_{\ell} \cap L_i$.
Let $T_v$ denote the resulting copy of $\Tk$ in $H$.
Also every leaf of $\Tk$ has some label $(1,j)$ with $1\leq j\leq 2^{k-1}$ and thus 
the corresponding leaf in $T_v$ is contained in $M_{(1,j,\ell)}$. 
By Property~\ref{decomp:firstneighbours}
it follows that the latter is adjacent to $x$ in $H$. 
Thus the claim follows. \done

\medskip

From now on, for every $v\in L_k = \bigcup_{\ell \in [4]}M_{(k,1,\ell)} $ fix a tree $T_v$ as described above. 
Since, for the property that we are aiming for,
we need to have control on how many such trees become useless when some edge is removed from $H$,
we define 
\begin{equation}\label{def:Se1}
S_e := \left\{
v\in L_k: ~ 
e\in E(T_v)\cup \{xw:~ w~ \text{is a leaf of }T_v \}
\right\}
\end{equation}
for every edge $e\in E(H)$.
Under assumption of the event $\E_H$, we next deduce that $S_e$ does not get too large.

\begin{clm}\label{cl:Se}
Let $\E_H$ hold. Then $|S_e|\leq n^{2/3-\eps}$
for every $e\in E(H)$. 
\end{clm}

\textit{Proof.}
Let $e\in E(H)$, then
$e$ is incident to at least one vertex $y\in M_{(i,j,\ell)}$
for some $(i,j,\ell)\in \Ik$, because of Property~\ref{decomp:edges}.
For every vertex $v\in S_e$ there must exist a path $P$ in $T_v$ leading from $y$ to $v$. According to Property~\ref{tree:children}, this path $P$ needs to be
of the form
$
P=(y,v_{i+1},v_{i+2},\ldots,v_{k-1},v)
$ 
with $v_s\in L_s$ for every $i+1\leq s\leq k-1$. 
Following Property~\ref{decomp:degreebound} and Property~\ref{decomp:edges},
we have $d_H(v_s,L_{s+1})\leq n^{(\alpha_{s+1}-\alpha_s)\eps}$
for every $i+1\leq s\leq k-1$ and
$d_H(y,L_{i+1})\leq n^{(\alpha_{i+1}-\alpha_i)\eps}$.
Thus, the number of all possible such $y$-$v$-paths $P$ that
belong to some $T_v$ with $v\in L_k$ is bounded from above
by
$$
\prod_{s=i}^{k-1}
n^{(\alpha_{s+1}-\alpha_s)\eps}
=
n^{(\alpha_{k}-\alpha_i)\eps} \leq n^{\frac23-\eps},
$$
where the last inequality holds by (\ref{alpha}) and since $\alpha_i\geq \alpha_1=1$. This proves the claim. \done

\medskip

We next expose the edges incident to
$r\in R$. By a standard Chernoff argument (apply Lemma~\ref{lem:Chernoff2} and union bound), we conclude that with probability at least $1-\exp(-0.5\ln^2 n)$ it holds that
\begin{equation}\label{eq:deg_Se}
d_G(r,S_e) < \ln^2 n
\end{equation}
for every $e\in E(H)$, and 
\begin{equation}\label{eq:deg_M}
d_G\left(r, M_{(k,1,\ell)} \right) > \frac{1}{2}p |M_{(k,1,\ell)}| 
\stackrel{\tref{decomp:size} }{=} 
\frac{n^{-\frac{1}{3} + (\alpha_k+1) \eps}}{2 \ln^{2\alpha_k} n}
\stackrel{(\ref{alpha})}{=}
\frac{n^{\frac{1}{3} + \eps}}{2 \ln^{2\alpha_k} n}
\end{equation}

for every $\ell\in [4]$.
Conditioning on~(\ref{eq:deg_Se}) and (\ref{eq:deg_M}) as well as the event $\E_H$,
which together hold with probability at least 
$1-\exp(-0.5\ln^{1.5} n)\geq 1-n^{-3}$, it remains to prove that
for every subgraph 
$B$ with $e(B)\leq n^{1/3}\ln n$ 
the graph $H\setminus B \subset G\setminus B$  
contains a copy $T$ 
of $\mathcal{T}_{k_1}$ as was described at the beginning
of the proof.

So, let $B$ of size 
$e(B)\leq n^{1/3}\ln n$ be given. 
Set $S_B:= \bigcup_{e\in B}S_e$
and observe that
$$
d_{G\setminus B}\left(r, S_B \right)
<d_{G}\left(r, S_B \right)
\stackrel{(\tref{eq:deg_Se})}{<}
|B|\cdot \ln^2 n 
< n^{\frac13}\ln^3 n~ .
$$
For every $\ell \in [4]$ we thus obtain
\begin{align*}
d_{G\setminus B} \left( r, M_{(k,1,\ell)} \right) 
 \geq  
d_{G} \left( r, M_{(k,1,\ell)} \right) - e(B) 
	\stackrel{(\ref{eq:deg_M})}{>} 
  \frac{n^{\frac13+\eps}}{2\ln^{2\alpha_k} n} -  
  n^{\frac13}\ln n
 > d_{G\setminus B}\left(r, S_B \right)~ ,
\end{align*}
provided $n$ is large enough.
Hence, for every $\ell \in [4]$, we find a vertex 
$x_{\ell} \in N_{G\setminus B}(r)$ with
$x_{\ell} \in M_{(k,1,\ell)} \setminus S_B$.
By the choice of $T_{x_{\ell}}$ (Claim~\ref{cl:Tv}), 
the tree $T_{x_{\ell}}$ 
is a copy of $\Tk$ in $H$ with $x_{\ell}$ being its root,
such that every leaf of $T_{x_{\ell}}$ 
is adjacent to $x$ in $H\subset G$ and such that
$V(T_{x_{\ell}})\subset M_{\ell}$.
Moreover, by the definition of $S_e$ in (\ref{def:Se1}), and since
$x_{\ell}\notin S_B$, we find that
$$E(T_{x_{\ell}})\cup \left\{xw:~ w\text{ is a leaf of $T_{x_{\ell}}$}  \right\}\subset E(H)\setminus E(B) ~ .$$ 
Now, set $T$ to be the union
of $T_{x_1}$, $T_{x_2}$ and $\{x_1r,x_2r\}$. Then, 
since $M_{1}\cap M_2 = \varnothing$, 
we get that
$T\subset G\setminus B$ is a copy of $\Tkplus = \mathcal{T}_{k_1}$ with all its leaves being adjacent to $x$ in $G\setminus B$. Moreover, $r$ is the root of $T$,
and we have $x\notin V(T)$, since $x\notin M_{\ell} \supset V(T_{x_{\ell}})$.
\end{proof}

\medskip

\subsection{Finding good structures -- Part II}

\begin{proof}[Proof of Lemma~\ref{lem:tec2_connector}]
Let $\delta>0$ be given. Fix $k\in\mathbb{N}$ such that
$\delta\geq \eps := 1/(9\cdot 2^{k-2}-3)$
holds, and let $k_2=k$. Again, it is enough to prove the lemma for
$p=n^{-2/3+\eps}$ and to use the monotonicity
of the desired property. Again,
we set
$$\Ik=\left\{(i,j,\ell)\in \mathbb{N}^3:~ 
1\leq i\leq k,~ 1\leq j\leq 2^{k-i},~ 1\leq \ell\leq 4
\right\}~ .$$

Let $V=[n]$ be the vertex set,  
let $A\subset V$ with $|A|=n^{1/3}$ be fixed,
and let $x\in V\setminus A$.
Before exposing the edges of $G\sim G_{n,p}$
let 
$$
[n]=\{x\}\cup \bigcup_{(i,j,\ell)\in \Ik} V_{(i,j,\ell)} \cup R~ 
$$
be any partition of the vertex set satisfying
$A\subset R$ and $|V_{(i,j,\ell)}|=n/(2^{k+4})$
for every $(i,j,\ell)\in \Ik$. We  will show that with probability
at least $1-n^{-2}$ a random graph $G\sim G_{n,p}$ is such that for every vertex set $M\subset V\setminus \{x\}$ and every subgraph $B$, with 
$e(B)\leq n^{2/3}\ln n$ 
and $d_B(v)\leq \ln^2 n$ for every $v\in V\setminus M$,
the graph $G\setminus B$ 
contains a vertex $z$ and four binary trees $T_{\ell}$ 
as described by the lemma, additionally satisfying that $z\in R\setminus A$ and $V(T_{\ell})\subset V_{\ell}:=\bigcup_{(i,j)} V_{(i,j,\ell)}$ for every $\ell\in [4]$.
Then, taking union bound over all choices of $x$, 
Lemma~\ref{lem:tec2_connector} follows immediately.

\medskip

In order to do so, we first expose the edges of $G$ on $[n]\setminus R$.
By Lemma~\ref{lem:Decomp} we know that with probability at least $1-\exp(-\ln^{1.5} n)$ there exists a subgraph $H\subset G$ with $V(H)\subset [n]\setminus R$ as well as subsets 
$M_{(i,j,\ell)}\subset V_{(i,j,\ell)} $ satisfying 
the Properties~\ref{decomp:inclusion}--\ref{decomp:edges}. 
Let $\E_H$ be the event that such a graph $H$ with vertex sets $M_{(i,j,\ell)}$ exists.
From now on, we will condition on $\E_H$ to hold. 
Following Claim~\ref{cl:Tv} we then know that
for every $\ell\in [4]$ and every vertex 
$v\in M_{(k,1,\ell)}$ there exists a tree $T_v \cong \Tk$ in $H$
such that $V(T_v)\subset V_{\ell}$
and such that all the leaves of $T_v$ are adjacent to $x$ in~$H$.

\medskip

Set $R'=R\setminus A$ and for every $Q\subset L_k = \bigcup_{\ell \in [4]} M_{(k,1,\ell)}$
let 
\begin{equation}\label{def:big}
Big_Q : =\left\{ u\in R':~ d_G(u,Q)> n^{\frac13+\frac{\eps}{2}} \right\}~ .
\end{equation}
Next expose all remaining edges of $G$. Under the assumption of $\E_H$
we then have the following

\begin{clm}\label{cl:Big}
Let $\E_H$ hold.
Then, with probability at least $1-\exp(-0.5\ln^2 n)$ the following holds:
\begin{enumerate}
\item[(a)] 
$\forall \ell\in [4] ~ \forall v\in R':
~ d_G(v,M_{(k,1,\ell)}) > n^{1/3+2\eps/3}
$,
\item[(b)]
$
|N_G(A)\cap R'| > n^{2/3+2\eps/3}
$,
\item[(c)]
for every $Q\subset L_k$ of size $n^{1-2\eps}$ we have
$
|Big_Q|\leq n^{2/3+\eps/2}
$.
\end{enumerate}
\end{clm}
 
\textit{Proof.}
For (a) notice that $d_G(v,M_{(k,1,\ell)})\sim Bin(|M_{(k,1,\ell)}|,p)$
and thus 
$$\Exp(d_G(v,M_{(k,1,\ell)})) = |M_{(k,1,\ell)}|p
	\stackrel{ \tref{decomp:size}  }{=}
	\frac{n^{-\frac{1}{3}+(\alpha_k+1)\eps}}{\ln^{2\alpha_k} n}
	\stackrel{(\ref{alpha})}{=} \frac{n^{\frac{1}{3}+\eps}}{\ln^{2\alpha_k} n}~ .$$
Applying Chernoff's inequality (Lemma~\ref{lem:Chernoff1}) and a union bound over all choice of 
$\ell\in [4]$ and $v\in R'$ we get
$\Prob(\text{(a) fails}) \leq 4n \exp(-n^{1/3})~ . $

For (b), observe first that $|R'|> \frac{n}{4}$. 
Applying Chernoff's inequality (Lemmas~\ref{lem:Chernoff1}
and~\ref{lem:Chernoff2}) and a union bound,
we see that with probability at least
$1-\exp(-0.9\ln^2 n)$ we have
$$
e_G(A,R')\geq \frac{1}{2}|A||R'|p > \frac{n^{\frac23+\eps}}{8}
~~~ 
\text{and}
~~~
d_G(v,A)<\ln^2 n ~~~ \text{for every }v\in R'~ .
$$
From these inequalities we can conclude that  
$$
|N_G(A)\cap R'| \geq \frac{e_G(A,R')}{\ln^2 n}  > n^{\frac23(1+\eps)}~ .
$$
Thus, $\Prob(\text{(b) fails}) \leq \exp(-0.9\ln^2 n)$.

For (c) we start by observing that according to a
Chernoff and union bound argument 
(applying Lemma~\ref{lem:Chernoff2})
the following property
holds with probability at least $1-\exp(-n)$:
for every subsets $X\subset L_k$ and $Y\subset R'$ of sizes
$|X|=n^{1-2\eps}$ and $|Y|=n^{2/3+\eps/2}$
we have $e_G(X,Y)< n^{1+\eps/2}$.
Given that property,
assume that (c) fails to hold. Then there
exist $Q\subset L_k$ of size $n^{1-2\eps}$
and a subset $Big_Q'\subset Big_Q$
with $|Big_Q'|=n^{2/3+\eps/2}$. 
We then have $e_G(Q,Big_Q')<n^{1+\eps/2}$
under assumption of the property mentioned above;
but also $e_G(Q,Big_Q') > n^{1/3+\eps/2} \cdot |Big_Q'| = n^{1+\eps}$
according to the definition of $Big_Q$ in (\ref{def:big}), a contradiction.
Thus, $\Prob(\text{(c) fails}) \leq \exp(-n)$.

Finally, summing up all the failure probabilities, that were obtained above, the claim follows.~\done

\medskip

Conditioning on the event $\E_H$
and on the properties described in Claim~\ref{cl:Big},
it remains to prove that
for every vertex set $M\subset V\setminus \{x\}$ and every subgraph $B$, with 
$e(B)\leq n^{2/3}\ln n$ 
and $d_B(v)\leq \ln^2 n$ for every $v\in V\setminus M$,
the graph $G\setminus B$ 
contains a vertex $z$ and four binary trees $T_{\ell}$ 
as described at the beginning of the proof.
So, let any such $M$ and $B$ be given. Similarly to the proof of 
Lemma~\ref{lem:tec1_connector} consider the set
\begin{align*}
S_e & = \left\{
v\in L_k: ~ 
e\in E(T_v)\cup \{xw:~ w~ \text{is a leaf of }T_v \}
\right\} 
\end{align*}
for every edge $e\in E(H)$, and let
$S_X:=\bigcup_{e\in X}S_e$ for every $X\subset V$. 
Further, let
\begin{align*}
B_i  :=\left\{
e\in B\cap H:~ e=ab ~ \text{with }a\in L_{i-1}\setminus M
~ \text{and } b\in L_i
\right\} ~~ \text{and} ~~
B^{\ast}  := \bigcup_{i\in [k]} B_i~ .
\end{align*}
Similarly to 
Claim~\ref{cl:Se}, we prove the following

\begin{clm}
Let $\E_H$ hold. Then
$|S_{B^{\ast}}|\leq n^{1-2\eps}$.
\end{clm}

\textit{Proof.}
We first bound $|S_{B_i}|$ for every $i\in [k]$. If $v$ is a vertex in $S_{B_i}$,
then this means that the edge set $E(T_v)\cup \{xw:~ w~ \text{is a leaf of }T_v\}$
needs to contain an edge $e=ab\in B$ between a vertex $a\in L_{i-1}\setminus M$
and a vertex $b\in L_i$. By assumption on $B$ we have
$d_B(a) \leq \ln^2 n$. Moreover,
there must exist a path $P$ in $T_v$ leading from $b$ to $v$ which is
of the form 
$
P=(b,v_{i+1},v_{i+2},\ldots,v_{k-1},v)
$ 
with $v_s\in L_s$ for every $i+1\leq s\leq k-1$. 
Analogously to the proof of Claim~\ref{cl:Se} we know that
the number of such paths is at most
$
\prod_{s=i}^{k-1}
n^{(\alpha_{s+1}-\alpha_s)\eps}
=
n^{(\alpha_{k}-\alpha_i)\eps}
$.
Provided $n$ is large enough, we thus conclude that for $i\geq 2$ it holds that
\begin{align*}
|S_{B_i}| 
& < |L_{i-1}|\cdot \ln^2 n \cdot n^{(\alpha_{k}-\alpha_i)\eps}
= 
2^{k-i+3} n^{\frac13+\alpha_{i-1}\eps} \ln^2 n \cdot n^{\frac23-\alpha_i \eps} \\
& < n^{1-(\alpha_i-\alpha_{i-1})\eps} \ln^3 n
\leq n^{1-3\eps} \ln^3 n~ 
\end{align*}
where for the equation we make use of
Definition~(\ref{def:Li}), Property~\ref{decomp:size}, the equation $\alpha_k\eps = \frac{2}{3}$
from (\ref{alpha}), and
where in the last inequality we use that,
according to (\ref{def:alpha}),
$\alpha_i-\alpha_{i-1} = 3\cdot 2^{i-2} \geq 3$ holds.
Moreover, since $L_0=\{x\}$ and using (\ref{def:alpha}) again,
$$
|S_{B_1}| < |L_{0}|\cdot \ln^2 n \cdot n^{(\alpha_{k}-\alpha_1)\eps} = n^{\frac23-\eps} \ln^2 n~ .
$$

Hence, $|S_{B^{\ast}}|\leq \sum_{i\in [k]} |S_{B_i}| < n^{1-2\eps}$, which finishes the proof of the claim.~\done

\medskip

Now, under assumption of $|S_{B^{\ast}}|\leq n^{1-2\eps}$ and the properties in Claim~\ref{cl:Big}
it follows that
\begin{align*}
|(N_{G\setminus B}(A)\cap R')\setminus Big_{S_{B^{\ast}}}|
& \geq 
|N_{G}(A)\cap R'| - e(B) - |Big_{S_{B^{\ast}}}| \\
& \stackrel{\text{(b),(c)}}{>}
n^{\frac23(1+\eps)} - n^{\frac23}\ln n - n^{\frac23+\frac{\eps}{2}}
> 2e(B)~ ,
\end{align*}
provided $n$ is large enough. Thus, there exists a vertex 
$$
z\in (N_{G\setminus B}(A)\cap R')\setminus Big_{S_{B^{\ast}}}
$$
such that $d_B(z)=0$. In particular, we then obtain that
\begin{align*}
d_{G\setminus B}(z,M_{(k,1,\ell)})
=
d_{G}(z,M_{(k,1,\ell)})  
\stackrel{\text{(a)}}{>}
n^{\frac13(1+2\eps)}
> 
d_G(z,S_{B^{\ast}})
\end{align*}
for every $\ell\in [4]$,
where in the last inequality we use that
$z\notin Big_{S_{B^{\ast}}}$.
For every $\ell\in [4]$ 
we thus find a vertex $r_{\ell}\in M_{(k,1,\ell)}\setminus S_{B^{\ast}}$
such that $zr_{\ell} \in E(G\setminus B)$.
The latter already ensures that Property~\ref{good2:top} holds.
By the choice of $T_{\ell}:=T_{r_{\ell}}$
(Claim~\ref{cl:Tv}), we know that
$T_{\ell}$ is a copy of $\Tk$ in $H$
such that $V(T_{\ell})\subset M_{\ell}$
and such that all the leaves of $T_{\ell }$ are adjacent to $x$ in $H$.
Since $r_{\ell}\notin S_{B^{\ast}}$ we also know that
the set $E(T_{\ell})\cup \left\{xw:~ w\text{ is a leaf of }T_{\ell}\right\}$
does not contain any edges from $B_i$ for $i\in [k]$.
Thus, if there is an edge $e=ab\in B$ that belongs to $E(T_{\ell})$ 
with $a\in L_{i-1}$ and $b\in L_i$ for some $2\leq i\leq k$,
then $a\in M$ must hold, according to the definition of
$B_i$, $B^\ast$ and $S_{B^\ast}$. This yields Property~\ref{good2:edges}.
Analogously, if there were edges from $B$ incident to $x$ and
to a leaf of $T_{\ell}$, then $x\in M$ would need to hold.
However, we have $x\notin M$ by assumption, and thus
Property~\ref{good2:leaves} follows.
Finally, \ref{good2:x} holds, since
$V(T_{\ell})\subset M_{\ell}$ by Property~\ref{tree:VE},
and $x\notin M_{\ell}$.
\end{proof}

\section{Concluding remarks}\label{sec:concluding}

{\bf Adding more constraints.}
Another variant of Maker-Breaker games are 
Walker-Breaker games (see e.g.~\cite{CT2016}, \cite{EFKP2014}
and~\cite{FM2019b}) 
which put even more constraints on the edges that Maker
may choose from in every round. 
Here, Maker is only allowed to claim her edges
according to a walk. That is, in each round she must claim a free edge or she must walk along one of her already claimed edges, 
such that this edge is incident to her current position
in the graph.
It is quite natural to ask what happens in the connectivity game
on $G\sim G_{n,p}$ when the game is played in the Walker-Breaker setting.
This is work in progress already. Moreover, one may consider the variant in which Breaker also needs to play as a Walker, as suggested in~\cite{EFKP2014} and~\cite{FM2019a}. 
We have not considered this variant, but we would be interested in
how it behaves compared to the usual Maker-Breaker game
and the Connector-Breaker game, respectively.\\

{\bf Considering different graph properties.}
In our paper we consider the Connector-Breaker game
on $G\sim G_{n,p}$ in which Connector aims for a spanning tree. 
By combining our argument
with the randomized strategy given by Ferber, Krivelevich and Naves~\cite{FKN2015},
we can even show that $n^{-2/3+o(1)}$ is the size of the threshold probability when Connector aims for a Hamilton cycle.
For a clearer presentation in this paper, 
we however skip the full argument here. A proof will appear in a follow-up paper.
Furthermore, it would be interesting to consider other graph properties
and to study the relation between the Connector-Breaker game, the Walker-Breaker game and their Maker-Breaker analogue.
For example, consider the $H$-game where Maker (or Connector/Walker) wins if she claims all the edges of a copy of a given (constant size) graph $H$. Following~\cite{BL2000} and the approach given in~\cite{CT2016}
it turns our that for all the three types of games
the threshold bias for a $(1:b)$ game played on $K_n$ is of the same order, namely $\Theta(n^{1/m_2(H)})$, with $m_2(H)$ being
the maximum $2$-density of $H$. This is in contrast to the connectivity game discussed in this paper.
We wonder whether in the unbiased $H$-game on $G\sim G_{n,p}$
it also holds that the threshold probabilities
for winning either variant are of the same order.

\medskip

{\bf Acknowledgement.}
This project was started as a part of the Bachelor thesis of the second author who would like to thank the group of Anusch Taraz for providing a nice working environment and fruitful discussions.

\end{document}